\documentclass[11pt,reqno]{amsart}
\usepackage{graphicx}
\usepackage{hyperref}
\usepackage{amssymb}
\usepackage{cite}
\usepackage{amsmath}
\usepackage{latexsym}
\usepackage{amscd}
\usepackage{amsthm}
\usepackage{mathrsfs}
\usepackage{url}

\numberwithin{equation}{section}

\newtheorem{thm}{Theorem}

\newtheorem{lem}[thm]{Lemma}
\newtheorem{prop}[thm]{Proposition}
\newtheorem{conj}[thm]{Conjecture}

\theoremstyle{definition}

\theoremstyle{remark}
\newtheorem{rem}[thm]{Remark}

\def\R{\mathbb R}

\def\N{\mathbb N}

\def\SS{\mathbb S}

\def\pt{\partial}

\begin{document}
\title[On the spectra of three Steklov eigenvalue problems]{On the spectra of three Steklov eigenvalue problems on warped product manifolds}
\author{Changwei~Xiong}
\address{Mathematical Sciences Institute, Australian National University, Canberra, ACT 2601, Australia}
\email{changwei.xiong@anu.edu.au}
\date{\today}
\thanks{This research was supported by Australian Laureate Fellowship FL150100126 of the Australian Research Council. The author would like to express his sincere gratitude to Ben~Andrews for stimulating discussions and valuable suggestions.}
\subjclass[2010]{{35P15}, {58C40}}
\keywords{Spectrum; Steklov eigenvalue problem; Eigenvalue estimate; Warped product manifold}

\maketitle

\begin{abstract}
Let $M^n=[0,R)\times \mathbb{S}^{n-1}$ be an $n$-dimensional ($n\geq 2$) smooth Riemannian manifold equipped with the warped product metric $g=dr^2+h^2(r)g_{\mathbb{S}^{n-1}}$ and diffeomorphic to a Euclidean ball. Assume that $M$ has strictly convex boundary. First, for the classical Steklov eigenvalue problem, we obtain an optimal lower (upper, respectively) bound for its spectrum in terms of $h'(R)/h(R)$ when $Ric_g\geq 0$ ($\leq 0$, respectively). Second, for  two fourth-order Steklov eigenvalue problems studied by Kuttler and Sigillito in 1968, we derive a lower bound for their spectra in terms of either $h'(R)/h^3(R)$ or $h'(R)/h(R)$ when $Ric_g\geq 0$, which is optimal for certain cases; in particular, we confirm a conjecture raised by Q.~Wang and C.~Xia for warped product manifolds of dimension $n=2$ or $n\geq 4$. For some proofs we utilize the Reilly's formula and reveal a new feature on its use.
\end{abstract}

\section{Introduction}\label{sec1}

One of the most important and extensively-studied topics in differential geometry is the estimate for various kinds of eigenvalues. Well-known eigenvalue problems include the closed Laplacian eigenvalue problem, Dirichlet eigenvalue problem and Neumann eigenvalue problem. Compared with these eigenvalue problems, the Steklov eigenvalue problem received less attention in the past. However, recently there has been increasing interest in the estimate for the Steklov eigenvalue problem, especially since Fraser and Schoen's work \cite{FS11}. In this paper we are concerned with estimates for the spectra of three types of Steklov eigenvalue problems. Note that generally if the parameter (or the eigenvalue) appears on the boundary of a Riemannian manifold, the problem is called a Steklov (-type) eigenvalue problem.

Let $(M^n,g)$ be an $n$-dimensional ($n\geq 2$) connected compact smooth Riemannian manifold with boundary $\pt M$. In the first part of this paper, we consider the classical Steklov eigenvalue problem, introduced by Steklov in 1895 (see \cite{KKK14}, \cite{Ste02}):
\begin{equation}\label{S1}
\begin{cases}
\Delta \varphi=0,&\text{ in } M,\\
\dfrac{\partial \varphi}{\partial \nu}=\sigma \varphi, &\text{ on } \pt M,
\end{cases}
\end{equation}
where $\Delta$ denotes the Laplace--Beltrami operator of $M$ and $\nu$ is the outward unit normal along $\pt M$. Equivalently, the Steklov eigenvalues form the spectrum of the Dirichlet-to-Neumann map $\Lambda:C^\infty(\pt M)\rightarrow C^\infty(\pt M)$ defined by
\begin{equation*}
\Lambda \varphi=\frac{\pt (\mathcal{H}\varphi)}{\pt \nu},\: \varphi\in C^\infty(\pt M),
\end{equation*}
where $\mathcal{H}\varphi$ is the harmonic extension of $\varphi$ to the interior of $M$. The Dirichlet-to-Neumann map $\Lambda$ is a first-order elliptic pseudodifferential operator \cite[pp. 37--38]{Tay96} and its spectrum is nonnegative, discrete and unbounded (counted with multiplicity):
\begin{equation*}
0=\sigma_0<\sigma_1\leq \sigma_2\leq \cdots \nearrow +\infty.
\end{equation*}
For later use, we denote by $\sigma_{(m)}$ the eigenvalues without counting multiplicity. For instance, for the $n$-dimensional Euclidean ball $B_R$ with radius $R$, we have
\begin{equation*}
\sigma_{(0)}=\sigma_0=0,\quad \sigma_{(1)}=\sigma_1=\cdots=\sigma_{n}=\frac{1}{R},
\end{equation*}
and $\sigma_{(m)}=m/R$ with multiplicity $C_{n+m-1}^{n-1}-C_{n+m-3}^{n-1}$ for $m\geq 2$ (see e.g. \cite{GP14}). Besides, the eigenvalue $\sigma_k$ of the problem \eqref{S1} has the variational characterization:
\begin{equation}\label{V1}
\sigma_k=\inf_{\substack{\varphi\in H^1(M), \: \varphi|_{\pt M}\neq 0\\  \int_{\pt M}\varphi \varphi_i da_g=0,\: i=0,\cdots, k-1}}\frac{\int_M|\nabla \varphi|^2dv_g}{\int_{\pt M} \varphi^2da_g},
\end{equation}
where $H^1(M)$ denotes the standard Sobolev space, $\varphi_i$ is the $i$th eigenfunction, and $dv_g$ and $da_g$ stand for the volume element and the area element of $M$ and $\pt M$, respectively.

There is an extensive literature concerning the Steklov eigenvalue problem~\eqref{S1}. We refer to the recent survey \cite{GP14} and the references therein for an account of this topic. In particular, let us mention some of the motivations (cf. \cite{CEG11,GP14,Esc97,Esc00}) for investigating the Steklov eigenvalue problem \eqref{S1}. First, the Steklov eigenvalue problem can be used as a model for Electrical Impedance Tomography, for the Dirichlet-to-Neumann map is intimately related to the Calder\'{o}n problem \cite{Cal80} on determining the anisotropic conductivity of a body from current and voltage measurements on its boundary. Second, in heat transmission, the eigenfunction $\varphi$ stands for the steady temperature on $M$ with the flux on the boundary $\sigma$-proportional to the temperature. Third, when on a two-dimensional manifold, the Steklov eigenvalues can be viewed as the squares of the natural frequencies of a vibrating free membrane with its mass concentrated on its boundary with constant density (see \cite{LP15}). Fourth, the Steklov eigenvalue problem is also useful in fluid mechanics (see \cite{FK83,KK01}). Last, in view of the variational characterization \eqref{V1} of the first nonzero eigenvalue $\sigma_1$, a sharp lower bound for $\sigma_1$ would imply a sharp Sobolev trace inequality for $\varphi\in H^1(M)$,
\begin{equation}\label{ST}
\sigma_1\int_{\pt M} |\varphi-\bar{\varphi}|^2 da_g\leq \int_M |\nabla \varphi|^2 dv_g,
\end{equation}
where $\bar{\varphi}:=\int_{\pt M} \varphi da_g/|\pt M|_g$ is the average of $\varphi$ on the boundary.

Due to the above backgrounds, there have been many interesting problems on the estimate for the Steklov eigenvalues. Among them the following conjecture was proposed by J.~Escobar \cite{Esc99} in 1999.
\begin{conj}[J.~Escobar \cite{Esc99}]\label{conj1}
Let $(M^n,g)$ $(n\geq 3)$ be a connected compact smooth Riemannian manifold with boundary. Assume that $Ric_g\geq 0$ and that the principal curvatures of the boundary $\pt M$ are bounded below by a constant $c>0$. Then the first nonzero Steklov eigenvalue $\sigma_1$ has a lower bound
\begin{equation*}
\sigma_1\geq c,
\end{equation*}
with equality only for the Euclidean ball of radius $1/c$.
\end{conj}
For $n=2$, the above result was proved by L.~E.~Payne \cite{Pay70} in 1970 for the Euclidean case, and by J.~Escobar \cite{Esc97} in 1997 for the Riemannian case (assuming the Gaussian curvature $K\geq 0$). For $n\geq 3$, J.~Escobar \cite{Esc97} obtained the nonsharp lower bound $\sigma_1>c/2$ in 1997 by use of the Reilly's formula. Also for $n\geq 3$, Monta\~{n}o  confirmed Conjecture~\ref{conj1} for a ball equipped with rotationally invariant metric \cite{Mon13} and for Euclidean ellipsoids \cite{Mon16}.

Motivated by Escobar's Conjecture~\ref{conj1}, in this paper we first consider to estimate the spectrum of the Steklov problem~\eqref{S1} in terms of the boundary curvature of the manifold under the condition on its Ricci curvature. For other works of a similar flavour, see e.g. \cite{PS19,Xio18,CGH18}. To obtain optimal estimates, we restrict ourselves to the special case where $M^n$ is a warped product manifold. Let $M^n=[0,R)\times \SS^{n-1}$ be a smooth Riemannian manifold equipped with the warped product metric
\begin{equation*}
g=dr^2+h^2(r)g_{\SS^{n-1}}.
\end{equation*}
Note that A.~Kasue \cite{Kas83} and R.~Ichida \cite{Ich81}, independently, showed that if a Riemannian manifold $M$ has nonnegative Ricci curvature and (weakly) mean convex boundary $\pt M$, then either $\pt M$ is connected, or $M$ is isometric to a Riemannian product manifold $\Gamma\times [0,a]$ (with constant warping function). So in view of the setting of Escobar's conjecture, the warped product manifold $M$ which we work on is of only one boundary component. Thus we need impose $h(0)=0$ and $M$ is a topological ball. Moreover, to guarantee that the metric $g$ is smooth at the origin, we need additional conditions on the derivatives of $h(r)$ at $r=0$ (see Section~4.3.4 in Petersen's book \cite{Pet16}). To sum up, let us agree to the following condition on $h$ throughout the paper.
\begin{itemize}
  \item[(A)] $h\in C^\infty([0,R))$, $h(r)>0$ for $r\in (0,R)$, $h'(0)=1$ and $h^{(2k)}(0)=0$ for all integers $k\geq 0$.
\end{itemize}

Now we are ready to state our first main result, i.e., we prove an optimal lower bound for the spectrum of the Steklov eigenvalue problem~\eqref{S1} when $M$ has nonnegative Ricci curvature and strictly convex boundary.
\begin{thm}\label{thm1}
Let $M^n=[0,R)\times \SS^{n-1}$ be an $n$-dimensional ($n\geq 2$) smooth Riemannian manifold equipped with the warped product metric
\begin{equation*}
g=dr^2+h^2(r)g_{\SS^{n-1}},
\end{equation*}
where the warping function $h$ satisfies Assumption (A). Suppose that $M$ has nonnegative Ricci curvature and strictly convex boundary. Then the $m$th Steklov eigenvalue $\sigma_{(m)}$ of the problem \eqref{S1} without counting multiplicity satisfies
\begin{equation}\label{eq1.4}
\sigma_{(m)}\geq m\frac{h'(R)}{h(R)},\quad m\geq 0,
\end{equation}
with equality if and only if $h(r)=r$, or $M$ is isometric to the Euclidean ball with radius $R$.
\end{thm}
Note that the boundary of $M$, the slice $\{R\}\times \SS^{n-1}$ in $M$, is totally umbilical with principal curvatures equal to $h'(R)/h(R)$. In view of Escobar's Conjecture~\ref{conj1}, the lower bound in terms of $h'(R)/h(R)$ as in \eqref{eq1.4} is quite natural, which may inspire the investigation on general Riemannian manifolds. As mentioned above, Theorem~\ref{thm1} has been proved by Escobar \cite{Esc97} for $n=2$ and $m=1$, and by Monta\~{n}o \cite{Mon13} for $n\geq 3$ and $m=1$. Both proofs in \cite{Esc97} and \cite{Mon13} are different from ours. Our proof is more direct.

We can also obtain a parallel result for the case $Ric_g\leq 0$.
\begin{thm}\label{thm1.1}
Assumptions as in Theorem~\ref{thm1} except $Ric_g\geq 0$ replaced by $Ric_g\leq 0$. Then the $m$th Steklov eigenvalue $\sigma_{(m)}$ of the problem \eqref{S1} without counting multiplicity satisfies
\begin{equation}
\sigma_{(m)}\leq m\frac{h'(R)}{h(R)},\quad m\geq 0,
\end{equation}
with equality if and only if $h(r)=r$, or $M$ is isometric to the Euclidean ball with radius $R$.
\end{thm}
Theorem~\ref{thm1.1} has been proved by Monta\~{n}o \cite{Mon13a} for $n\geq 2$ and $m=1$ using a different argument.

In the second part of this paper, we consider a fourth-order Steklov eigenvalue problem, which was initially investigated by J.~R.~Kuttler and V.~G.~Sigillito \cite{KS68} in 1968:
\begin{equation}\label{S2}
\begin{cases}
\Delta^2 \varphi=0,\quad &\text{in }M,\\
\dfrac{\pt \varphi}{\pt \nu}=0,\quad \dfrac{\pt(\Delta \varphi)}{\pt \nu}+\xi \varphi=0,\quad &\text{on }\pt M.
\end{cases}
\end{equation}
Here the constant $\xi$ denotes the eigenvalue. The eigenvalue problem \eqref{S2} is important in biharmonic analysis and elastic mechanics. In particular, in two-dimensional case the eigenfunction $\varphi$ represents the deformation of the linear elastic supported plate $M$ under the action of the transversal exterior force $f(x)=0$, $x\in M$ subject to Neumann boundary condition (see \cite{Vil97,TG70,XW18}). In addition, the first nonzero eigenvalue $\xi_1$ arises as an optimal constant in an a priori inequality (see \cite{Kut82}). The eigenvalues of the problem \eqref{S2} form a discrete and increasing sequence (counted with multiplicity):
\begin{equation*}
0=\xi_0<\xi_1\leq \xi_2\leq \cdots \nearrow +\infty.
\end{equation*}
We also use $\xi_{(m)}$ to denote the eigenvalues without counting multiplicity. For the $n$-dimensional Euclidean ball $B_R$ with radius $R$, we have $\xi_{(m)}=m^2(n+2m)/R^3$ with multiplicity $C_{n+m-1}^{n-1}-C_{n+m-3}^{n-1}$ (see \cite[Theorem~1.5]{XW18}). In addition, the eigenvalue $\xi_k$ has the variational characterization:
\begin{equation*}
\xi_k=\inf_{\substack{ \varphi\in H^2(M), \: \pt_\nu \varphi|_{\pt M}=0,\:\varphi|_{\pt M}\neq 0\\ \int_{\pt M}\varphi \varphi_i da_g=0,\: i=0,\cdots, k-1}}\frac{\int_M(\Delta \varphi)^2dv_g}{\int_{\pt M} \varphi^2da_g},
\end{equation*}
where $\varphi_i$ is the $i$th eigenfunction.

As a motivation for our work, let us mention the following conjecture on the sharp lower bound of the first nonzero eigenvalue $\xi_1$ of the Steklov problem \eqref{S2} proposed by Qiaoling~Wang and Changyu~Xia \cite{WX18}.
\begin{conj}[Q.~Wang and C.~Xia \cite{WX18}]\label{conj2}
Let $(M^n,g)$ $(n\geq 2)$ be a connected compact smooth Riemannian manifold with boundary. Assume that $Ric_g\geq 0$ and that the principal curvatures of the boundary $\pt M$ are bounded below by a constant $c>0$. Denote by $\lambda_1=\lambda_1(\pt M)$ the first nonzero eigenvalue of the Laplacian of $\pt M$. Then the first nonzero Steklov eigenvalue $\xi_1$ has a lower bound
\begin{equation}\label{bound-conj2}
\xi_1\geq \frac{n+2}{n-1}c\lambda_1,
\end{equation}
with equality only for the Euclidean ball of radius $1/c$.
\end{conj}
The reason why Wang and Xia proposed Conjecture~\ref{conj2} may lie in the fact that they \cite{XW13} proved the nonsharp lower bound $\xi_1>nc\lambda_1/(n-1)$ in 2013 using the Reilly's formula. We remark that unlike Escobar's Conjecture~\ref{conj1}, the presence of the nonlocal term $\lambda_1$ in the lower bound \eqref{bound-conj2} may increase the difficulty to solve the problem.

In this paper we are able to confirm Conjecture~\ref{conj2} for warped product manifolds $M^n=[0,R)\times \SS^{n-1}$ of dimension $n=2$ or $n\geq 4$. In fact, we provide a lower bound for the spectrum of the Steklov problem \eqref{S2}.
\begin{thm}\label{thm2}
Let $M^n=[0,R)\times \SS^{n-1}$ be an $n$-dimensional ($n\geq 2$) smooth Riemannian manifold equipped with the warped product metric
\begin{equation*}
g=dr^2+h^2(r)g_{\SS^{n-1}},
\end{equation*}
where the warping function $h$ satisfies Assumption (A). Suppose that $M$ has nonnegative Ricci curvature and strictly convex boundary. Denote by $\xi_{(m)}$ the $m$th eigenvalue of the Steklov problem \eqref{S2} without counting multiplicity and set $\tau_m=m(n-2+m)$. Then for $n=2$ and $m\geq 1$, we have
\begin{equation}\label{bound2}
\xi_{(m)}\geq 2m^2(m+1)\frac{h'(R)}{h^3(R)}.
\end{equation}
For $n=3$ and $m\geq 2$, we have
\begin{align}\label{bound3}
\xi_{(m)}&\geq \frac{(4\tau_m-13)\tau_m}{2\tau_m-6}\frac{h'(R)}{h^3(R)}.
\end{align}
For $n\geq 4$ and $m\geq 1$, we have
\begin{equation}\label{bound4}
\xi_{(m)}\geq \frac{(4\tau_m+n(n-4))\tau_m}{2\tau_m+(n-1)(n-4)}\frac{h'(R)}{h^3(R)}.
\end{equation}
Moreover, the equality holds for $n=2$ and $m\geq 1$, or for $n\geq 4$ and $m=1$, if and only if $h(r)=r$, or $M$ is isometric to the Euclidean ball with radius $R$.
\end{thm}
As mentioned above, as a corollary of Theorem~\ref{thm2}, Wang and Xia's Conjecture~\ref{conj2} holds for warped product manifolds of dimension $n=2$ or $n\geq 4$. Precisely, for a general warped product manifold $M^n=[0,R)\times \SS^{n-1}$ with warping function $h$, the principal curvatures of the slice $\{R\}\times \SS^{n-1}$ are equal to $h'(R)/h(R)$ and $\lambda_1(\{R\}\times \SS^{n-1})=(n-1)/h^2(R)$. So the lower bound in \eqref{bound2} or \eqref{bound4} for $m=1$ is exactly the one in \eqref{bound-conj2}.

Lastly, we are interested in another fourth-order Steklov eigenvalue problem, which was initially studied by J.~R.~Kuttler and V.~G.~Sigillito \cite{KS68} in 1968 and by L.~E.~Payne \cite{Pay70} in 1970:
\begin{equation}\label{S3}
\begin{cases}
\Delta^2 \varphi=0,\quad &\text{in }M,\\
\varphi=0,\quad \Delta \varphi=\eta  \dfrac{\pt \varphi}{\pt \nu},\quad & \text{on }\pt M.
\end{cases}
\end{equation}
Here the constant $\eta$ stands for the eigenvalue. The eigenvalue problem \eqref{S3} has some backgrounds in the theory of elasticity and in conductivity as well. See the Introduction in \cite{FGW05} for an interesting interpretation of the boundary condition of \eqref{S3} in the theory of elasticity. Similar to the classical Steklov eigenvalue problem~\eqref{S1}, the problem~\eqref{S3} is also closely related to inverse problems in partial differential equations (see \cite{Cal80}), for in this case the set of the eigenvalues of the problem~\eqref{S3} is the same as that of the Neumann-to-Laplacian map for biharmonic equations; see e.g. \cite{Liu11} for more details. In addition, the first eigenvalue $\eta_0$ is of significance since as observed by Kuttler \cite{Kut72,Kut79} it is the sharp constant for $L^2$ a priori estimates for the Laplace equation with nonhomogeneous Dirichlet boundary conditions. See e.g. \cite{BFG09,BGM06,RS15,Liu11,Liu16,WX09} for related works. The eigenvalues of the problem \eqref{S3} constitutes a discrete and increasing sequence (counted with multiplicity):
\begin{equation*}
0<\eta_0<\eta_1\leq \eta_2\leq \cdots \nearrow +\infty.
\end{equation*}
Note that the first eigenvalue $\eta_0$ is positive and simple (see \cite[Theorem~1]{BGM06} or \cite{RS15}). We also use $\eta_{(m)}$ to denote the eigenvalues without counting multiplicity. For the $n$-dimensional Euclidean ball $B_R$ with radius $R$, we know $\eta_{(m)}=(n+2m)/R$ with multiplicity $C_{n+m-1}^{n-1}-C_{n+m-3}^{n-1}$ (see \cite[Theorem~1.3]{FGW05}). The $k$th eigenvalue $\eta_k$ of the problem~\eqref{S3} admits the variational characterization:
\begin{equation}\label{V3}
\eta_k=\inf_{\substack{ \varphi\in H^2(M),\: \varphi|_{\pt M}=0,\:\pt_\nu \varphi|_{\pt M}\neq 0\\  \int_{\pt M}\varphi \varphi_i da_g=0,\: i=0,\cdots, k-1}}\frac{\int_M(\Delta \varphi)^2dv_g}{\int_{\pt M} (\pt_\nu \varphi)^2da_g},
\end{equation}
where $\varphi_i$ is the $i$th eigenfunction.

Our argument for Theorem~\ref{thm2} allows us to prove parallel results for the eigenvalue problem~\eqref{S3}.
\begin{thm}\label{thm3}
Let $M^n=[0,R)\times \SS^{n-1}$ be an $n$-dimensional ($n\geq 2$) smooth Riemannian manifold equipped with the warped product metric
\begin{equation*}
g=dr^2+h^2(r)g_{\SS^{n-1}},
\end{equation*}
where the warping function $h$ satisfies Assumption (A). Suppose that $M$ has nonnegative Ricci curvature and strictly convex boundary. Denote by $\eta_{(m)}$ the $m$th eigenvalue of the Steklov problem \eqref{S3} without counting multiplicity and set $\tau_m=m(n-2+m)$. Then for $n=2$ and $m\geq 1$, we have
\begin{equation}\label{bound5}
\eta_{(m)}\geq 2(m+1)\frac{h'(R)}{h(R)}.
\end{equation}
For $n=3$ and $m\geq 2$, we have
\begin{align}\label{bound6}
\eta_{(m)}&\geq \frac{4\tau_m-13}{\tau_m-3}\frac{h'(R)}{h(R)}.
\end{align}
For $n\geq 4$ and $m\geq 1$, we have
\begin{equation}\label{bound7}
\eta_{(m)}\geq \frac{(4\tau_m+n(n-4))(n-1)}{2\tau_m+(n-1)(n-4)}\frac{h'(R)}{h(R)}.
\end{equation}
Moreover, the equality holds for $n=2$ and $m\geq 1$, or for $n\geq 4$ and $m=1$, if and only if $h(r)=r$, or $M$ is isometric to the Euclidean ball with radius $R$.
\end{thm}
We remark that our proof also works for $n\geq 2$ and $m=0$; the conclusion simply reads  $\eta_{(0)}=\eta_0\geq nh'(R)/h(R)$ with rigidity statement. However, since the result for $m=0$, i.e., for the first eigenvalue $\eta_0$ has been proved in \cite{WX09} for a general setting, we choose not to state it in Theorem~\ref{thm3}; see \cite{RS15} for an improvement of \cite{WX09}. In addition, we should point out that the lower bounds in \eqref{bound6} and \eqref{bound7} are interesting only for small $m$.

The proofs of Theorem~\ref{thm1}, Theorem~\ref{thm2} and Theorem~\ref{thm3} mainly consist of two steps. In Step~1 we obtain the characterization of all the eigenfunctions in the problem by separation of variables. Thus all the eigenfunctions are of the simple form $\varphi(r,p)=\psi(r)\omega(p)$, $r\in [0,R)$ and $p\in \SS^{n-1}$, where $\psi(r)$ satisfies certain ODE and $\omega(p)$ is some spherical harmonic on $\SS^{n-1}$. In Step~2, for Theorem~\ref{thm1} in all dimensions, or Theorems~\ref{thm2} and \ref{thm3} in dimension $n=2$, we can directly analyze the resulting ODE to conclude the proof; while for Theorems~\ref{thm2} and \ref{thm3} in dimension $n\geq 3$, we need to make best use of the Reilly's formula \cite{Rei77} to finish the proof.

For the proofs involving the Reilly's formula, we find a new and interesting feature on the use of this formula. More precisely, instead of throwing away the Ricci integral term in the Reilly's formula (as done in most of the literature), we need separate a nontrivial positive term from it to balance the negative term. See Remark~\ref{key-remark} in Section~\ref{sec4.3}. This kind of process seems impossible for general Riemannian manifolds, which may indicate that Wang and Xia's Conjecture~\ref{conj2} in its full generality (at least for $n\geq 3$) would be much difficult.

The structure of this paper is as follows. In Section~\ref{sec2} we collect some basic facts on the warped product manifolds, recall the Reilly's formula which will be used later, and review the representation of spherical harmonics in terms of the harmonic homogeneous polynomials. In Sections~\ref{sec3}, \ref{sec4} and \ref{sec5} we prove Theorems~\ref{thm1}, \ref{thm2} and \ref{thm3}, respectively. At the end of Section~\ref{sec4}, we also discuss briefly the remaining case $n=3$ and $m=1$, and the case $Ric_g\leq 0$ for Theorem~\ref{thm2}. In the Appendix we provide some computation results. For the notation in the remaining part of this paper, sometimes we write $B_R$ for the warped product manifold $M=[0,R)\times \SS^{n-1}$ and we omit the integral element $dr$. And as far as a spherical harmonic $\omega$ is concerned, we assume that it is normalized, i.e., $\int_{\SS^{n-1}}\omega^2 da=1$.

\section{Preliminaries}\label{sec2}
\subsection{Ricci curvature and the principal curvatures on the boundary}

Let $M^n=[0,R)\times \SS^{n-1}$ be an $n$-dimensional ($n\geq 2$) smooth Riemannian manifold equipped with the warped product metric
\begin{equation*}
g=dr^2+h^2(r)g_{\SS^{n-1}},
\end{equation*}
where the warping function $h$ satisfies
\begin{itemize}
  \item[(A)] $h\in C^\infty([0,R))$, $h(r)>0$ for $r\in (0,R)$, $h'(0)=1$ and $h^{(2k)}(0)=0$ for all integers $k\geq 0$.
\end{itemize}

The Ricci curvature of the warped product manifold $M$ reads
\begin{align*}
Ric_g&=-\left(\frac{h''}{h}-(n-2)\frac{1-(h')^2}{h^2}\right)g-(n-2)\left(\frac{h''}{h}+\frac{1-(h')^2}{h^2}\right)dr^2.
\end{align*}
In fact, denoting by $\{\pt_r,E_1,\dots,E_{n-1}\}$ an orthonormal basis of $T_{(r,p)}M$ at $(r,p)$ ($r>0$), we have the sectional curvatures of $M$ given by (see \cite{Pet16})
\begin{align*}
K(\pt_r \wedge E_i)&=-\frac{h''(r)}{h(r)},\quad i=1,\dots,n-1,\\
K(E_i\wedge E_j)&=\frac{1-(h'(r))^2}{h^2(r)},\quad i\neq j.
\end{align*}
When $n=2$, Ricci curvature reduces to the Gaussian curvature or the sectional curvature.

On the other hand, it is well-known that the boundary $\pt M=\{R\}\times \SS^{n-1}$  of $M$ is totally umbilical with principal curvatures
\begin{equation*}
\kappa_1=\cdots=\kappa_{n-1}=\frac{h'(R)}{h(R)}.
\end{equation*}

For the proofs of Theorems~\ref{thm1}, \ref{thm2} and \ref{thm3}, we need the following lemma concerning the property of the warping factor.
\begin{lem}
Under the conditions of Theorems~\ref{thm1}, \ref{thm2} or \ref{thm3}, we have
\begin{equation}\label{eq2.3}
h''(r)\leq 0, \text{ and } 0<h'(r)\leq 1,\quad r\in [0,R).
\end{equation}
\end{lem}
\begin{proof}
Note that the eigenvalues of $Ric_g$ are
\begin{equation*}
-(n-1)\frac{h''}{h}, -\left(\frac{h''}{h}-(n-2)\frac{1-(h')^2}{h^2}\right),\cdots,-\left(\frac{h''}{h}-(n-2)\frac{1-(h')^2}{h^2}\right),
\end{equation*}
and so $Ric_g\geq 0$ is equivalent to
\begin{equation*}
h''\leq 0,\text{ and } h''\leq (n-2)\frac{1-(h')^2}{h}.
\end{equation*}
Meanwhile, notice that $h'(0)=1$, and $h'(R)>0$ by the strict convexity of the boundary. Combining $h''(r)\leq 0$, we know that $0<h'(r)\leq 1$ for $r\in [0,R)$.
\end{proof}

For the proof of Theorem~\ref{thm1.1}, the inequalities in \eqref{eq2.3} are reversed.

\subsection{Reilly's formula}
For an $n$-dimensional connected compact smooth Riemannian manifold $(M^n,g)$ with boundary and any smooth function $f\in C^{\infty}(M)$, we have the Reilly's formula (\cite{Rei77}):
\begin{align}
&\int_{M}\left((\Delta f)^2-|\nabla^2 f|^2-Ric_g(\nabla f,\nabla f)\right)dv_g\nonumber \\
&=\int_{\pt M}\left((H\pt_\nu f+2\Delta^\pt u)\pt_\nu f+II(\nabla^\pt u,\nabla^\pt u)\right)da_g.\label{Rei}
\end{align}
Here $u=f|_{\pt M}$, the symbols $\Delta^\pt$ and $\nabla^\pt$ are the Laplace--Beltrami operator and the connection on the boundary with respect to the induced metric, respectively. Moreover, $II$ and $H=tr_gII$ denote the second fundamental form and the mean curvature of the boundary with respect to the outer unit normal $\nu$, respectively. The proof of the Reilly's formula is by integrating the following Bochner's formula on $M$,
\begin{equation}\label{Boc}
\frac{1}{2}\Delta(|\nabla f|^2)=|\nabla^2 f|^2+g(\nabla f,\nabla(\Delta f))+\mathrm{Ric}_{g}(\nabla f,\nabla f),
\end{equation}
using divergence theorem to get some boundary integral terms, and arranging suitably these terms to obtain \eqref{Rei}.

\subsection{Spherical harmonics}

Given a spherical harmonic $\omega$ on $\SS^{n-1}$ of degree $m\geq 0$, it can be viewed as the restriction on $\SS^{n-1}$ of a harmonic homogeneous polynomial $\tilde{\omega}$ on $\R^n$ of the same degree $m$. For each $m\geq 0$, let $\mathcal{D}_m$ denote the space of harmonic homogeneous polynomials on $\R^n$ of degree $m$ and $\mu_m$ be the dimension of $\mathcal{D}_m$. For example, we know
\begin{align*}
&\mathcal{D}_0=span\{1\},\quad \mu_0=1,\\
&\mathcal{D}_1=span\{x_i,\:i=1,\cdots,n\},\quad \mu_1=n,\\
&\mathcal{D}_2=span\{x_ix_j,\: x_1^2-x_k^2,\:1\leq i<j\leq n,\: 2\leq k\leq n\},\quad \mu_2=\frac{n^2+n-2}{2},
\end{align*}
and $\mu_m=C_{n+m-1}^{n-1}-C_{n+m-3}^{n-1}$ for $m\geq 2$. See \cite{ABR92} for basic facts concerning $\mathcal{D}_m$ and $\mu_m$.

For a spherical harmonic $\omega$ on $\SS^{n-1}$ of degree $m\geq 0$, one of its basic properties is that $-\Delta_{\SS^{n-1}}\omega=\tau_m\omega$ with $\tau_m=m(n-2+m)$.

\section{Proofs of Theorems~\ref{thm1} and \ref{thm1.1}}\label{sec3}
\subsection{The characterization of the Steklov eigenfunctions}
The Steklov eigenvalue problem we consider in this section is
\begin{equation}\label{S3.1}
\begin{cases}
\Delta \varphi=0,&\text{ in } M,\\
\dfrac{\partial \varphi}{\partial \nu}=\sigma \varphi, &\text{ on } \pt M,
\end{cases}
\end{equation}
and the variational characterization of its eigenvalues reads
\begin{equation}\label{V3.2}
\sigma_k=\inf_{\substack{ \varphi\in H^1(M),\: \varphi|_{\pt M}\neq 0\\  \int_{\pt M}\varphi \varphi_i da_g=0,\: i=0,\cdots, k-1}}\frac{\int_M|\nabla \varphi|^2dv_g}{\int_{\pt M} \varphi^2da_g},
\end{equation}
where $\varphi_i$ is the $i$th eigenfunction.

First we obtain the characterization of all its eigenfunctions on warped product manifolds in Theorem~\ref{thm1} by separation of variables. The following result for the first nontrivial eigenfunction was proved in \cite[Lemma~3]{Esc00}. Here we follow the approach in \cite{Esc00}. For completeness, we include the proof here.
\begin{prop}\label{prop1}
For the warped product manifold $M$ in Theorem~\ref{thm1}, any nontrivial eigenfunction $\varphi$ of the problem \eqref{S3.1} can be written as $\varphi(r,p)=\psi(r)\omega(p)$, where $\omega$ is a spherical harmonic on $\SS^{n-1}$ of some degree $m\geq 1$, i.e.,
\begin{equation*}
-\Delta_{\SS^{n-1}}\omega=\tau_m\omega \text{ on }\SS^{n-1},\quad \tau_m=m(n-2+m),
\end{equation*}
and $\psi$ is a nontrivial solution of the ODE
\begin{equation*}
\begin{cases}
\dfrac{1}{h^{n-1}}\dfrac{d}{dr}(h^{n-1}\dfrac{d}{dr}\psi)-\dfrac{\tau_m\psi}{h^2}=0,\quad &r\in (0,R),\\
\psi(0)=0.&
\end{cases}
\end{equation*}
For any nontrivial solution $\psi$ to the above ODE, the $m$th eigenvalue $\sigma_{(m)}$ without counting multiplicity is given by $\sigma_{(m)}=\psi'(R)/\psi(R)$.
\end{prop}
\begin{proof}
We use separation of variables. Note that the space $L^2(B_R)$ is equivalent to the space $L^2((0,R))\otimes L^2(\SS^{n-1})$. Take $\{\omega_k\}$, $k=0,1,\cdots$, as a complete orthonormal basis of $L^2(\SS^{n-1})$ which is a set of spherical harmonics on $\SS^{n-1}$. That is,
\begin{equation*}
-\Delta_{\SS^{n-1}}\omega_k=\tau_{m(k)}\omega_k,\quad \tau_m=m(n-2+m).
\end{equation*}
We arrange $\omega_k$ such that $\omega_0$ is of degree $m(0)=0$; $\{\omega_k\}_{k=1}^n$ are of degree $m(k)=1$; etc.

Let $\psi_0=1$. For $k\geq 1$, let $\psi_k\neq 0$ solve
\begin{equation*}
\begin{cases}
\dfrac{1}{h^{n-1}}\dfrac{d}{dr}(h^{n-1}\dfrac{d}{dr}\psi_k)-\dfrac{\tau_{m(k)}\psi_k}{h^2}=0,\quad &r\in (0,R),\\
\psi_k(0)=0,\quad \psi_k(R)=1.&
\end{cases}
\end{equation*}
Here $\psi_k(0)=0$ is needed because the function $\psi_k(r)\omega_k(p)$ below is supposed to be continuous at the origin, and the condition $\psi_k(R)=1$ is imposed to specify the solution. Now define $\varphi_k(r,p)=\psi_k(r)\omega_k(p)$, $k\in \N$. We claim that
\begin{equation*}
{span} \{\varphi_k,k\in \N\}=\{\varphi\in C^\infty(B_R):\Delta\varphi=0 \text{ in } B_R\}.
\end{equation*}
To prove the claim, note that for $k\geq 0$ we have
\begin{equation*}
\Delta \varphi_k=\left(\frac{1}{h^{n-1}}\frac{d}{dr}(h^{n-1}\frac{d}{dr}\psi_k)-\frac{\tau_{m(k)}\psi_k}{h^2}\right)\omega_k=0.
\end{equation*}

Now take any smooth $\bar{\varphi}$ with $\Delta\bar{\varphi}=0$ in $B_R$. Since the basis $\{\omega_k\}$, $k=0,1,\dots$, for $L^2(\SS^{n-1})$ is complete, we can first decompose
\begin{equation*}
\bar\varphi|_{\pt B_R}=\sum_{k=0}^\infty c_k \omega_k,
\end{equation*}
for a sequence of constants $c_k$. Then the function $f:=\bar{\varphi}-\sum_{k=0}^\infty c_k \varphi_k$ satisfies
\begin{equation*}
\Delta f=0, \text{ in }B_R, \quad f=0,\text{ on }\pt B_R,
\end{equation*}
which implies $f=0$, or $\bar{\varphi}=\sum_{k=0}^\infty c_k \varphi_k$. So we have proved the claim. In particular, any eigenfunction $\varphi$ can be written as
\begin{equation*}
\varphi=\sum_{k=0}^\infty c_k \varphi_k.
\end{equation*}
The eigenfunction corresponding to $\sigma_0$ is $\varphi_0=const\neq 0$. Let $\varphi$ be an eigenfunction corresponding to $\sigma_1>0$. Since $\int_{\pt B_R} \varphi da_g=0$, we know $c_0=0$. Using the variational characterization for $\sigma_1$, we have
\begin{align*}
\sigma_1&=\frac{\sum_{k\geq 1}c_k^2\int_{B_R}|\nabla \varphi_k|^2 dv_g}{\sum_{k\geq 1}c_k^2\int_{\pt B_R}\varphi_k^2 da_g}\\
&=\frac{\sum_{k\geq 1}c_k^2\int_{0}^R \left((\psi_k')^2+\tau_{m(k)}\frac{\psi_k^2}{h^2}\right)h^{n-1}dr}{\sum_{k\geq 1}c_k^2\int_{\pt B_R}\varphi_k^2 da_g}.
\end{align*}
Since the numerator on the right-hand side is strictly increasing in $\tau_{m}$, we  conclude that the first nonzero eigenvalue $\sigma_{(1)}=\sigma_1$ is of multiplicity $n$ and its corresponding eigenspace is spanned by $\{\psi_k\omega_k,k=1,\cdots,n\}$. Note that all the spherical harmonics $\{\omega_k,k=1,\cdots,n\}$ are of the same degree $m(k)=1$, and so $\psi_1=\cdots=\psi_n$.

Once we determine the eigenspace corresponding to $\sigma_{(1)}$, we can use the variational characterization \eqref{V3.2} to determine all the subsequent eigenspaces.

\end{proof}

\subsection{Proof of Theorem~\ref{thm1}}

We only consider the case $m\geq 1$. By Proposition~\ref{prop1}, the $m$th Steklov eigenvalue $\sigma_{(m)}$ without counting multiplicity is given by
\begin{equation*}
\sigma_{(m)}=\frac{\psi'(R)}{\psi(R)},
\end{equation*}
where $\psi$ solves
\begin{equation*}
\begin{cases}
\dfrac{1}{h^{n-1}}\dfrac{d}{dr}(h^{n-1}\dfrac{d}{dr}\psi)-\dfrac{\tau_m\psi}{h^2}=0,\quad & r\in (0,R),\\
\psi(0)=0.&
\end{cases}
\end{equation*}
First note that we may carry out the integration to get
\begin{equation*}
h^{n-1}(r)\psi'(r)=\tau_m\int_0^r h^{n-3}(t)\psi(t)dt.
\end{equation*}
Without loss of generality assume $\psi(r)>0$ in a small neighbourhood of the origin. Then we have $\psi'(r)>0$ and $\psi(r)> 0$ for $r\in (0,R)$. In fact, since $h(r)=r+o(r)$ as $r\rightarrow 0+$, we have
\begin{equation}\label{Asym}
\psi(r)=r^m+o(r^m),\text{ as } r\rightarrow 0+,
\end{equation}
up to a constant multiple. Here the other solution $r^{2-n-m}$ in the asymptotic sense, which is singular at $r=0$, has been ruled out.

Now Theorem~\ref{thm1} follows from the following proposition.
\begin{prop}\label{prop2}
Under the conditions of Theorem~\ref{thm1}, we have
\begin{equation*}
\sigma_{(m)}=\frac{\psi'(R)}{\psi(R)}\geq m\frac{h'(R)}{h(R)}.
\end{equation*}
\end{prop}
\begin{proof}
Let
\begin{equation*}
q(r):=h(r)\psi'(r)-mh'(r)\psi(r).
\end{equation*}
So $q(0)=0$. Next we have
\begin{align*}
q'&=h'\psi'+h\psi''-mh''\psi-mh'\psi'\\
&\geq h\psi''+(1-m)h'\psi'\\
&=:u(r),
\end{align*}
where we have used $h''\leq 0$ and $\psi\geq 0$.

Using the equation satisfied by $\psi$, we have
\begin{equation*}
u=h\psi''+(1-m)h'\psi'=\tau_m\frac{\psi}{h}-\frac{\tau_m}{m}h'\psi'.
\end{equation*}
Note that
\begin{equation*}
u(0)=\tau_m(1-\frac{1}{m})\psi'(0).
\end{equation*}
So in view of \eqref{Asym}, for either $m=1$ or $m\geq 2$ we have $u(0)=0$.

Next we deduce
\begin{align*}
\frac{1}{\tau_m}u'&=\frac{\psi'}{h}-\frac{\psi h'}{h^2}-\frac{1}{m}h''\psi'-\frac{1}{m}h'\psi''\\
&\geq \frac{\psi'}{h}-\frac{\psi h'}{h^2}-\frac{h'}{mh}\left(\tau_m\frac{\psi}{h}-(n-1)h'\psi'\right)\\
&=\left(\frac{1}{h}+(n-1)\frac{(h')^2}{mh}\right)\psi'-(1+\frac{\tau_m}{m})\frac{h'\psi}{h^2}\\
&=\left(\frac{1}{h}+(n-1)\frac{(h')^2}{mh}\right)\frac{m}{\tau_m h'}(\tau_m\frac{\psi}{h}-u)-(1+\frac{\tau_m}{m})\frac{h'\psi}{h^2}\\
&=\frac{m}{h^2h'}(1-(h')^2)\psi-\left(\frac{1}{h}+(n-1)\frac{(h')^2}{mh}\right)\frac{m}{\tau_m h'}u\\
&\geq -\left(\frac{1}{h}+(n-1)\frac{(h')^2}{mh}\right)\frac{m}{\tau_m h'}u.
\end{align*}
It follows that $u(r)\geq 0$, and consequently $q(r)\geq 0$. In particular, $q(R)\geq 0$, as required.

\end{proof}

\subsection{Proof of Theorem~\ref{thm1.1}}

The difference between Theorem~\ref{thm1} and Theorem~\ref{thm1.1} does not affect Proposition~\ref{prop1}. So to prove Theorem~\ref{thm1.1}, we only need to reverse all the inequalities in Proposition~\ref{prop2} by using $h''\geq 0$ and $h'\geq 1$ instead of $h''\leq 0$ and $0<h'\leq 1$.



\section{Proof of Theorem~\ref{thm2}}\label{sec4}

The fourth-order Steklov eigenvalue problem we consider in this section is:
\begin{equation}\label{S4.1}
\begin{cases}
\Delta^2 \varphi=0,\quad &\text{in }M,\\
\dfrac{\pt \varphi}{\pt \nu}=0,\quad \dfrac{\pt(\Delta \varphi)}{\pt \nu}+\xi \varphi=0,\quad & \text{on }\pt M.
\end{cases}
\end{equation}
Its $k$th eigenvalue $\xi_k$ has the variational characterization:
\begin{equation}\label{V4.2}
\xi_k=\inf_{\substack{ \varphi\in H^2(M),\:\pt_\nu \varphi|_{\pt M}=0,\: \varphi|_{\pt M}\neq 0\\  \int_{\pt M}\varphi \varphi_i da_g=0,\: i=0,\cdots, k-1}}\frac{\int_M(\Delta \varphi)^2dv_g}{\int_{\pt M} \varphi^2da_g},
\end{equation}
where $\varphi_i$ is the $i$th eigenfunction.

\subsection{The characterization of the Steklov eigenfunctions}
Similar to Proposition~\ref{prop1}, we may characterize the eigenfunctions of the problem~\eqref{S4.1} as follows.
\begin{prop}\label{prop3}
For the warped product manifold $M$ in Theorem~\ref{thm2}, any nontrivial eigenfunction $\varphi$ of the problem \eqref{S4.1} can be written as $\varphi(r,p)=\psi(r)\omega(p)$, where $\omega$ is a spherical harmonic on $\SS^{n-1}$ of some degree $m\geq 1$, i.e.,
\begin{equation*}
-\Delta_{\SS^{n-1}}\omega=\tau_m\omega \text{ on }\SS^{n-1},\quad \tau_m=m(n-2+m),
\end{equation*}
and $\psi$ is a nontrivial solution of the ODE
\begin{equation}\label{eq4}
\begin{cases}
\dfrac{1}{h^{n-1}}\dfrac{d}{dr}(h^{n-1}\dfrac{d}{dr}\psi)-\dfrac{\tau_{m}\psi}{h^2}=\tilde{\psi},&\\
\dfrac{1}{h^{n-1}}\dfrac{d}{dr}(h^{n-1}\dfrac{d}{dr}\tilde{\psi})-\dfrac{\tau_m\tilde\psi}{h^2}=0,&\\
\psi(0)=0,\quad  \psi'(R)=0,\quad \tilde{\psi}(0)=0.&
\end{cases}
\end{equation}
For any nontrivial solution $\psi$ to the above ODE, the $m$th eigenvalue $\xi_{(m)}$ without counting multiplicity is given by $\xi_{(m)}=-\tilde{\psi}'(R)/\psi(R)$.
\end{prop}
\begin{proof}
Take $\{\omega_k\}$, $k=0,1,\cdots$, to be a complete orthonormal basis of $L^2(\SS^{n-1})$ as in Proposition~\ref{prop1}. Let $\psi_0=1$. For $k\geq 1$, let $\psi_k$ and $\tilde{\psi}_k$ solve
\begin{equation}\label{eq5}
\begin{cases}
\dfrac{1}{h^{n-1}}\dfrac{d}{dr}(h^{n-1}\dfrac{d}{dr}\psi_k)-\dfrac{\tau_{m(k)}\psi_k}{h^2}=\tilde{\psi}_k,&\\
\dfrac{1}{h^{n-1}}\dfrac{d}{dr}(h^{n-1}\dfrac{d}{dr}\tilde{\psi}_k)-\dfrac{\tau_{m(k)}\tilde\psi_k}{h^2}=0,&\\
\psi_k(0)=0,\quad \psi_k(R)=1,\quad  \psi_k'(R)=0,\quad \tilde{\psi}_k(0)=0.&
\end{cases}
\end{equation}
Now define $\varphi_k(r,p)=\psi_k(r)\omega_k(p)$, $k\in \N$. We claim that
\begin{equation*}
{span} \{\varphi_k,k\in \N\}=\{\varphi\in C^\infty(B_R):\Delta^2\varphi=0 \text{ in } B_R,\: \partial_\nu \varphi=0 \text{ on } \pt B_R\}.
\end{equation*}
To prove the claim, note that for $k\geq 0$ we have
\begin{equation*}
\Delta \varphi_k=\left(\frac{1}{h^{n-1}}\dfrac{d}{dr}(h^{n-1}\dfrac{d}{dr}\psi_k)-\frac{\tau_{m(k)}\psi_k}{h^2}\right)\omega_k=\tilde{\psi}_k \omega_k,
\end{equation*}
and
\begin{equation*}
\Delta^2 \varphi_k=\left(\frac{1}{h^{n-1}}\dfrac{d}{dr}(h^{n-1}\dfrac{d}{dr}\tilde{\psi}_k)-\frac{\tau_{m(k)}\tilde\psi_k}{h^2}\right)\omega_k=0,
\end{equation*}
with $\pt_\nu \varphi_k=0$ on $\pt B_R$.

Now take any smooth $\bar{\varphi}$ with $\Delta^2\bar{\varphi}=0$ in $B_R$ and $\pt_\nu \bar{\varphi}=0$ on $\pt B_R$. Since the basis $\{\omega_k\}$, $k=0,1,\dots$, for $L^2(\SS^{n-1})$ is complete, we can first decompose
\begin{equation*}
\bar\varphi|_{\pt B_R}=\sum_{k=0}^\infty c_k \omega_k,
\end{equation*}
for a sequence of constants $c_k$. Then the function $f:=\bar{\varphi}-\sum_{k=0}^\infty c_k \varphi_k$ satisfies
\begin{equation*}
\Delta^2 f=0, \text{ in }B_R,\quad f=0 \text{ and } \pt_\nu f=0, \text{ on }\pt B_R,
\end{equation*}
which leads to $f=0$, or $\bar{\varphi}=\sum_{k=0}^\infty c_k \varphi_k$. So we have proved the claim. In particular, any eigenfunction $\varphi$ can be written as
\begin{equation*}
\varphi=\sum_{k=0}^\infty c_k \varphi_k.
\end{equation*}
It is easy to see that the first eigenvalue is $\xi_0=0$, corresponding to the constant eigenfunction $\varphi_0=const\neq 0$. Next we consider the eigenfunction $\varphi$ corresponding to $\xi_1=\xi_{(1)}$. Since $\int_{\pt B_R}\varphi da_g=0$, we have $c_0=0$.

For $k\geq 1$ define the energy
\begin{align*}
\beta_k:=-\frac{\tilde{\psi}_k'(R)}{\psi_k(R)}=\frac{\int_{B_R}(\Delta \varphi_k)^2 dv_g}{\int_{\pt B_R}\varphi_k^2 da_g}.
\end{align*}
It is easy to see $\beta_k\geq \xi_1$, for $ k\geq 1$.

On the other hand, we have
\begin{align*}
\xi_1&=\frac{\int_{B_R}(\Delta \varphi)^2 dv_g}{\int_{\pt B_R}\varphi^2 da_g}=\frac{\sum_{k\geq 1}c_k^2\int_{B_R}(\Delta \varphi_k)^2 dv_g}{\sum_{k\geq 1}c_k^2\int_{\pt B_R}\varphi_k^2 da_g}\\
&=\frac{\sum_{k\geq 1}c_k^2\int_{\pt B_R}\varphi_k^2 da_g\times \beta_k}{\sum_{k\geq 1}c_k^2\int_{\pt B_R}\varphi_k^2 da_g}\geq \xi_1.
\end{align*}
Therefore, for any $k\geq 1$ with $c_k\neq 0$, we have $\beta_k=\xi_1$. Now fix any $k\geq 1$ such that $\beta_k=\xi_1$. We claim that the corresponding $\tau_{m(k)}$ is equal to $\tau_1=n-1$. Then the eigenspace corresponding to $\xi_{(1)}$ can be determined as $span\{\varphi_k,k=1,\cdots,n\}$, and by the same argument the subsequent eigenspaces can also be dealt with.

To prove the claim, towards a contradiction assume $\tau_{m(k)}>\tau_1$. Note that
\begin{align*}
\beta_k=\frac{\int_0^R\left(\dfrac{1}{h^{n-1}}\dfrac{d}{dr}(h^{n-1}\dfrac{d}{dr}\psi_k)-\dfrac{\tau_{m(k)}\psi_k}{h^2}\right)^2h^{n-1}(r)dr}{\psi_k^2(R)h^{n-1}(R)}.
\end{align*}
First we prove that the numerator on the right-hand side of the above formula is strictly increasing in $\tau_{m(k)}$, or
\begin{align*}
&A:=\int_0^R\left(\frac{1}{h^{n-1}}\dfrac{d}{dr}(h^{n-1}\dfrac{d}{dr}\psi_k)-\frac{\tau_{m(k)}\psi_k}{h^2}\right)^2h^{n-1}dr\\
&>\int_0^R\left(\frac{1}{h^{n-1}}\dfrac{d}{dr}(h^{n-1}\dfrac{d}{dr}\psi_k)-\frac{\tau_1\psi_k}{h^2}\right)^2h^{n-1}dr=:B,
\end{align*}
and in particular $B$ would be finite. In fact, fixing any small $\varepsilon>0$ and setting
\begin{align*}
A(\varepsilon)&:=\int_\varepsilon^R\left(\frac{1}{h^{n-1}}\dfrac{d}{dr}(h^{n-1}\dfrac{d}{dr}\psi_k)-\frac{\tau_{m(k)}\psi_k}{h^2}\right)^2h^{n-1}dr,\\
B(\varepsilon)&:=\int_\varepsilon^R\left(\frac{1}{h^{n-1}}\dfrac{d}{dr}(h^{n-1}\dfrac{d}{dr}\psi_k)-\frac{\tau_1\psi_k}{h^2}\right)^2h^{n-1}dr,
\end{align*}
we obtain
\begin{align*}
&A(\varepsilon)-B(\varepsilon)\\
&=(\tau_{m(k)}-\tau_1)\int_\varepsilon^R\frac{\psi_k}{h^2}\left(-\frac{2}{h^{n-1}}\dfrac{d}{dr}(h^{n-1}\dfrac{d}{dr}\psi_k)+\frac{(\tau_{m(k)}+\tau_1)\psi_k}{h^2}\right)h^{n-1}dr\\
&=(\tau_{m(k)}-\tau_1)\left(2\int_\varepsilon^R\left(-\frac{\psi_k}{h^2}\dfrac{d}{dr}(h^{n-1}\dfrac{d}{dr}\psi_k)\right)dr+(\tau_{m(k)}+\tau_1)\int_\varepsilon^R (\psi_k)^2h^{n-5}dr\right).
\end{align*}
For the first term on the right-hand side, we have
\begin{align*}
&\int_\varepsilon^R\left(-\frac{\psi_k}{h^2}\dfrac{d}{dr}(h^{n-1}\dfrac{d}{dr}\psi_k)\right)dr\\
&=-\psi_kh^{n-3}\psi_k'\big|_{\varepsilon}^R+\int_\varepsilon^R\left((\frac{\psi_k}{h^2})'\times h^{n-1}\psi_k'\right)dr\\
&=\psi_kh^{n-3}\psi_k'\big|_{r=\varepsilon}+ \int_\varepsilon^R\left(h^{n-3}\left(\psi_k'-\frac{\psi_k}{h}h'\right)^2-h^{n-5}(\psi_k)^2(h')^2\right)dr\\
&\geq \psi_kh^{n-3}\psi_k'\big|_{r=\varepsilon}- \int_\varepsilon^Rh^{n-5}(\psi_k)^2dr,
\end{align*}
where we have used $(h')^2\leq 1$. Then we have
\begin{align*}
&A(\varepsilon)-B(\varepsilon)\geq (\tau_{m(k)}-\tau_1)\left(2\psi_kh^{n-3}\psi_k'\big|_{r=\varepsilon}+(\tau_{m(k)}+\tau_1-2) \int_\varepsilon^Rh^{n-5}(\psi_k)^2dr\right)\\
&\geq (\tau_{m(k)}-\tau_1)\left(2\psi_kh^{n-3}\psi_k'\big|_{r=\varepsilon}+(\tau_{m(k)}+\tau_1-2) \int_{R/2}^Rh^{n-5}(\psi_k)^2dr\right),
\end{align*}
as long as $\varepsilon\leq R/2$. Thus
\begin{align*}
&A-B=\lim_{\varepsilon\rightarrow 0+}(A(\varepsilon)-B(\varepsilon))\\
&\geq (\tau_{m(k)}-\tau_1)\left(2(\psi_k')^2(0)h^{n-2}(0)+(\tau_{m(k)}+\tau_1-2) \int_{R/2}^Rh^{n-5}(\psi_k)^2dr\right)>0.
\end{align*}
Here if $n=2$, we understand that $h^{n-2}(0)=1$.

 Then we can deduce that
\begin{align*}
\beta_k&>\frac{\int_0^R\left(\dfrac{1}{h^{n-1}}\dfrac{d}{dr}(h^{n-1}\dfrac{d}{dr}\psi_k)-\dfrac{\tau_1\psi_k}{h^2}\right)^2h^{n-1}(r)dr}{\psi_k^2(R)h^{n-1}(R)}\\
&=\frac{\int_{B_R}(\Delta (\psi_k\omega_1))^2 dv_g}{\int_{\pt B_R}(\psi_k\omega_1)^2 da_g}\geq \xi_1=\beta_k,
\end{align*}
which is a contradiction. Here we have used the function $f=\psi_k\omega_1$ as a test function in the variational characterization of $\xi_1$. To make it work, we need to verify that $f\in H^2(B_R)$, which can be proved by use of the Reilly's formula. Applying the Reilly's formula \eqref{Rei} to $f$ over $B_R\setminus B_\varepsilon$ for small $\varepsilon>0$, we get
\begin{align*}
&\int_{B_R\setminus B_\varepsilon}\left((\Delta f)^2-|\nabla^2 f|^2-Ric_g(\nabla f,\nabla f)\right)dv_g\\
&=\int_{\pt B_R\cup \pt B_\varepsilon}\left((H\pt_\nu f+2\Delta^\pt f)\pt_\nu f+II(\nabla^\pt f,\nabla^\pt f)\right)da_g.
\end{align*}
Note that when on $\pt B_\varepsilon$ we have $\nu=-\pt_r$. Then we may check that
\begin{align*}
&\lim_{\varepsilon\rightarrow 0+}\int_{\pt B_\varepsilon}\left((H\pt_\nu f+2\Delta^\pt f)\pt_\nu f+II(\nabla^\pt f,\nabla^\pt f)\right)da_g\\
&=\lim_{\varepsilon\rightarrow 0+}\left(\left(-(n-1)\frac{h'}{h}(-\psi_k')-2\psi_k\frac{\tau_1}{h^2}\right)(-\psi_k')-\frac{h'}{h}\frac{\tau_1\psi_k^2}{h^2}\right)h^{n-1}\bigg|_{r=\varepsilon}\\
&=\lim_{\varepsilon\rightarrow 0+}h^{n-2}\left(2\tau_1\frac{\psi_k\psi_k'}{h}-(n-1)h'(\psi_k')^2-\tau_1\frac{h'\psi_k^2}{h^2}\right)\\
&=h^{n-2}(0)\left(2\tau_1\frac{(\psi_k')^2(0)}{h'(0)}-(n-1)h'(0)(\psi_k'(0))^2-\tau_1\frac{(\psi_k'(0))^2}{h'(0)}\right)\\
&\geq 0.
\end{align*}
Then combining $Ric_g\geq 0$ in $B_R$ and $II>0$ along $\pt B_R$, we can conclude that
\begin{equation*}
\int_{B_R}|\nabla^2 f|^2 dv_g\leq \int_{B_R}(\Delta f)^2 dv_g<+\infty.
\end{equation*}

So we have the claim and we can determine the eigenspace corresponding to $\xi_{(1)}$. Next we can use the variational characterization \eqref{V4.2} for higher-order eigenvalues to determine the subsequent eigenspaces and finish the proof.

\end{proof}

By Proposition~\ref{prop3}, it suffices to prove that
\begin{equation*}
\xi_{(m)}=-\tilde{\psi}'(R)/{\psi(R)}=\dfrac{\int_{B_R}(\Delta (\psi \omega))^2 dv_g}{\int_{\pt B_R} (\psi \omega)^2 da_g}
\end{equation*}
has the lower bound in each case of Theorem~\ref{thm2}, where $\psi$ and $\tilde{\psi}$ solve \eqref{eq4} and $\omega$ is a spherical harmonic of degree $m$.

\subsection{Proof of Theorem~\ref{thm2} for $n=2$}\label{sec4.2}

First for general $n$ it is natural to use the change of variable
\begin{equation*}
s(r):=\int_{R/2}^r\frac{dt}{h^{n-1}(t)}.
\end{equation*}
So $s$ maps $(0,R]$ onto $(-\infty,s_0]$ for some $s_0>0$ and we obtain
\begin{align*}
s'(r)&=\frac{1}{h^{n-1}(r)},\quad  r'(s)=h^{n-1}(r(s)),\\
&\frac{d}{ds}=\frac{dr}{ds}\frac{d}{dr}=h^{n-1}(r)\frac{d}{dr}.
\end{align*}
Let $\tilde{b}(s):=\tilde{\psi}(r(s))$. Then the second equation of \eqref{eq4} reads
\begin{align*}
\tau_m h^{2(n-2)}&(r(s))\tilde b(s)=\tau_m h^{2(n-2)}(r)\tilde \psi(r)\\
&=h^{n-1}(r)\frac{d}{dr}\left(h^{n-1}(r)\frac{d}{dr}\tilde \psi(r)\right)=\frac{d^2}{ds^2}\tilde b(s).
\end{align*}

Now consider $n=2$. We have $\tilde{b}''(s)=m^2\tilde{b}(s)$, which implies that
\begin{equation*}
\tilde b(s)=c_1e^{ms}+c_2e^{-ms}.
\end{equation*}
It follows that
\begin{equation*}
\tilde{\psi}(r)=c_1e^{ms(r)}+c_2e^{-ms(r)}.
\end{equation*}
Note that as $r\rightarrow 0+$, we have $s(r)\rightarrow -\infty$. Therefore $\tilde{\psi}(0)=0$ implies $c_2=0$. Without loss of generality assume $c_1=1$, meaning
\begin{equation*}
\tilde\psi(r)=e^{ms(r)}.
\end{equation*}
Then the first equation of \eqref{eq4} becomes
\begin{equation*}
e^{ms(r)}h^2(r)=h^2(r)\psi''(r)+h(r)h'(r)\psi'(r)-m^2\psi(r).
\end{equation*}
Equivalently, using the $s$-variable and letting $b(s):=\psi(r(s))$, we have
\begin{equation}\label{eq3}
e^{ms}h^2(r(s))=b''(s)-m^2b(s),
\end{equation}
with the boundary conditions
\begin{equation*}
b(-\infty)=0,\quad b'(s_0)=0.
\end{equation*}

The general solution of Equation~\eqref{eq3} is
\begin{equation*}
b(s)=c_3e^{ms}+c_4e^{-ms}+\frac{1}{2m}e^{ms}\int_{-\infty}^s h^2(r(t))dt-\frac{1}{2m}e^{-ms}\int_{-\infty}^se^{2mt}h^2(r(t))dt.
\end{equation*}
The condition $b(-\infty)=0$ implies $c_4=0$. Then the condition $b'(s_0)=0$ forces $c_3$ to be
\begin{equation*}
c_3(s_0)=-\frac{1}{2m}\int_{-\infty}^{s_0} h^2(r(t))dt-\frac{1}{2m}e^{-2ms_0}\int_{-\infty}^{s_0}e^{2mt}h^2(r(t))dt.
\end{equation*}

Therefore, we have
\begin{align*}
b(s)=&\frac{1}{2m}e^{ms}\int_{s_0}^s h^2(r(t))dt-\frac{1}{2m}e^{m(s-2s_0)}\int_{-\infty}^{s_0}e^{2mt}h^2(r(t))dt\\
&-\frac{1}{2m}e^{-ms}\int_{-\infty}^se^{2mt}h^2(r(t))dt.
\end{align*}
In particular,
\begin{equation*}
b(s_0)=-\frac{1}{m}e^{-ms_0}\int_{-\infty}^{s_0}e^{2mt}h^2(r(t))dt.
\end{equation*}
Since $\tilde{\psi}'(r)=me^{ms(r)}s'(r)=me^{ms(r)}/h(r)$, to prove Theorem~\ref{thm2} for $n=2$ it suffices to prove
\begin{equation*}
-\frac{me^{ms(R)}}{\psi(R)}\geq 2m^2(m+1)\frac{h'(R)}{h^2(R)},\text{ or } -\frac{e^{ms_0}}{b(s_0)}\geq 2m(m+1)\frac{h'(r(s_0))}{h^2(r(s_0))}.
\end{equation*}
Note $b(s_0)<0$. Then we need to check that the function
\begin{equation*}
G(s):=\frac{e^{2ms}h^2(r(s))}{h'(r(s))}-2(m+1)\int_{-\infty}^{s}e^{2mt}h^2(r(t))dt
\end{equation*}
satisfies $G(s_0)\geq 0$. In fact, we can show the following result.
\begin{lem}\label{lem-G}
We have $G(s)\geq 0$ for $s\in (-\infty,s_0]$.
\end{lem}
\begin{proof}
First it is easy to check $G(-\infty)=0$. Next we have
\begin{align*}
G'(s)&=\frac{2me^{2ms}h^2(r(s))}{h'(r(s))}+\frac{e^{2ms}2h(r(s))h'(r(s))h(r(s))}{h'(r(s))}\\
&\quad -\frac{e^{2ms}h^3(r(s))h''(r(s))}{(h'(r(s)))^2}-2(m+1)e^{2ms}h^2(r(s))\\
&=2me^{2ms}h^2(r(s))\left(\frac{1}{h'(r(s))}-1\right)-\frac{e^{2ms}h^3(r(s))h''(r(s))}{(h'(r(s)))^2}\\
&\geq 0,
\end{align*}
where the last inequality is due to $h'(r(s))\in (0,1]$ and $h''(r(s))\leq 0$. So the conclusion follows immediately.

\end{proof}

So we have the lower bound \eqref{bound2} for $n=2$ and $m\geq 1$ in Theorem~\ref{thm2}. Moreover, when the equality holds, we must have $h''(r)\equiv 0$ and $h'(r)\equiv 1$, which means $h(r)=r$. So we complete the proof.

\subsection{Proof of Theorem~\ref{thm2} for $n\geq 3$}\label{sec4.3}

Let $\varphi(r,p)=\psi(r)\omega(p)$ be an eigenfunction corresponding to $\xi_{(m)}$. So we have
\begin{align*}
&\int_{B_R}(\Delta \varphi)^2=\xi_{(m)}\int_{\pt B_R}\varphi^2,\\
\int_{\pt B_R}II(\nabla^\pt \varphi,&\nabla^\pt \varphi)=\kappa \int_{\pt B_R}|\nabla^\pt \varphi|^2=\kappa \lambda_{(m)}\int_{\pt B_R}\varphi^2,
\end{align*}
where $\kappa=h'(R)/h(R)$ denotes the principal curvature of the boundary $\pt B_R$, and $\lambda_{(m)}=\tau_m/h^2(R)$ the $m$th eigenvalue of the Laplacian on the boundary $\pt B_R$ without counting multiplicity.

By the Reilly's formula~\eqref{Rei}, for fixed $c\in (0,1)$, we obtain
\begin{align*}
&c\xi_{(m)} \int_{\pt B_R}\varphi^2-\kappa \lambda_{(m)} \int_{\pt B_R}\varphi^2\\
&=\int_{B_R} |\nabla^2\varphi|^2-(1-c)\int_{B_R}(\Delta \varphi)^2+\int_{B_R} Ric(\nabla\varphi,\nabla\varphi).
\end{align*}
Our goal is to find $c$ as small as possible such that the right-hand side of the above formula is nonnegative.

By Proposition~\ref{prop-A} in the Appendix we have
\begin{align*}
\int_{B_R}&|\nabla^2\varphi|^2=\int_0^R \left((\psi'')^2+(n-1)\frac{(\psi')^2(h')^2}{h^2}\right.\\
&\left.+\tau\frac{2}{h^2}\left((\psi')^2+\psi^2\frac{(h')^2}{h^2}-3\frac{\psi\psi'h'}{h}\right)+\frac{\psi^2}{h^4}\tau(\tau-n+2)\right)h^{n-1},
\end{align*}
and
\begin{align*}
&\int_{B_R}(\Delta \varphi)^2=\int_0^R \left(\psi''+(n-1)\psi'\frac{h'}{h}-\tau\frac{\psi}{h^2}\right)^2h^{n-1}\\
&=\int_0^R \left((\psi'')^2+2\psi''\left((n-1)\psi'\frac{h'}{h}-\tau\frac{\psi}{h^2}\right)+\left((n-1)\psi_r\frac{h_r}{h}-\tau\frac{\psi}{h^2}\right)^2\right)h^{n-1}.
\end{align*}
Here and below we write $\tau=\tau_m$ for simplicity.

Therefore
\begin{align*}
&\int_{B_R} |\nabla^2\varphi|^2-(1-c)\int_{B_R}(\Delta \varphi)^2\\
&=\int_0^R \left(c(\psi'')^2-2(1-c)\psi''\left((n-1)\psi'\frac{h'}{h}
-\tau\frac{\psi}{h^2}\right)\right.\\
&\left.-(1-c)\left((n-1)\psi'\frac{h'}{h}-\tau\frac{\psi}{h^2}\right)^2+(n-1)\frac{(\psi')^2(h')^2}{h^2}\right.\\
&\left.+\tau\frac{2}{h^2}\left((\psi')^2+\psi^2\frac{(h')^2}{h^2}-3\frac{\psi\psi'h'}{h}\right)+\frac{\psi^2}{h^4}\tau(\tau-n+2)\right)h^{n-1}\\
&\geq \int_0^R \left(-\frac{1-c}{c}\left((n-1)\psi'\frac{h'}{h}-\tau\frac{\psi}{h^2}\right)^2+(n-1)\frac{(\psi')^2(h')^2}{h^2}\right.\\
&\left.+\tau\frac{2}{h^2}\left((\psi')^2+\psi^2\frac{(h')^2}{h^2}-3\frac{\psi\psi'h'}{h}\right)+\frac{\psi^2}{h^4}\tau(\tau-n+2)\right)h^{n-1}\\
&=:\int_0^R Kh^{n-1}.
\end{align*}
Set $b=1-c^{-1}$. Applying $1\geq (h')^2$ to the term $2\tau(\psi')^2/h^2$ in $K$, we have
\begin{align*}
K&\geq (b(n-1)^2+n-1+2\tau)\frac{(\psi')^2(h')^2}{h^2}-2\frac{\psi\psi'h'}{h^3}\left(b\tau(n-1)+3\tau\right)\\
&+(b\tau^2+\tau(\tau-n+2))\frac{\psi^2}{h^4}+2\tau\psi^2\frac{(h')^2}{h^4}\\
&\geq -\frac{\left(b\tau(n-1)+3\tau\right)^2}{b(n-1)^2+n-1+2\tau}\frac{\psi^2}{h^4}+(b\tau^2+\tau(\tau-n+2))\frac{\psi^2}{h^4}+2\tau\psi^2\frac{(h')^2}{h^4},
\end{align*}
where we have assumed $b(n-1)^2+n-1+2\tau>0$ for the last inequality holding.

Again by Proposition~\ref{prop-A} in the Appendix we have
\begin{align*}
&\int_{B_R} Ric(\nabla\varphi,\nabla\varphi)\geq \tau(n-2)\int_0^R (1-(h')^2)\frac{\psi^2}{h^4}h^{n-1}=:\int_0^R Lh^{n-1}.
\end{align*}
\begin{rem}\label{key-remark}
Here we separate a nontrivial positive term $\int_0^R Lh^{n-1}$ from $\int_{B_R} Ric(\nabla\varphi,\nabla\varphi)$ to conclude the proof below. Without this term we are unable to get the optimal lower bound for the case $n\geq 4$ and $m=1$. We believe this is a new feature on the use of the Reilly's formula, which may be applied to other problems.
\end{rem}
As a consequence, we get
\begin{align*}
K+L\geq &\left(-\frac{\left(b\tau(n-1)+3\tau\right)^2}{b(n-1)^2+n-1+2\tau}+b\tau^2+\tau(\tau-n+2)+\tau(n-2)\right)\frac{\psi^2}{h^4}\\
&-(n-4)\tau \frac{\psi^2(h')^2}{h^4}.
\end{align*}

\textbf{Case 1}: $n\geq 4$ and $m\geq 2$. Using $(h')^2\leq 1$ we have
\begin{align*}
&K+L\geq \left(-\frac{\left(b\tau(n-1)+3\tau\right)^2}{b(n-1)^2+n-1+2\tau}+b\tau^2+\tau(\tau-n+2)+2\tau\right)\frac{\psi^2}{h^4}.
\end{align*}
When
\begin{equation*}
b=\frac{-2\tau-n+4}{2\tau+(n-1)(n-4)}, \text{ or } c=\frac{2\tau+(n-1)(n-4)}{4\tau+n(n-4)},
\end{equation*}
we have $K+L\geq 0$. For this $b$, we can check $b(n-1)^2+n-1+2\tau>0$.

\textbf{Case 2}: $n\geq 4$ and $m= 1$. In this case we still choose
\begin{equation*}
b=\frac{-2\tau-n+4}{2\tau+(n-1)(n-4)}=-\frac{3}{n-1}, \text{ or } c=\frac{n-1}{n+2}.
\end{equation*}
For this $b$ we have $b(n-1)^2+n-1+2\tau=0$. However, we can directly check $K+L\geq 0$ to achieve our goal.

\textbf{Case 3}: $n= 3$ and $m\geq 2$. In this case we have
\begin{align*}
&K+L\geq \left(-\frac{\left(2b+3\right)^2\tau^2}{4b+2\tau+2}+(b+1)\tau^2\right)\frac{\psi^2}{h^4}.
\end{align*}
So when
\begin{equation*}
b=-\frac{2\tau-7}{2\tau-6}, \text{ or } c=\frac{2\tau-6}{4\tau-13},
\end{equation*}
we have $K+L\geq 0$. For this $b$, we can check $4b+2\tau+2=2(\tau-2)^2/(\tau-3)>0$.

In summary, for $n\geq 4$ and $m\geq 1$, we have
\begin{align*}
\xi_{(m)}\geq \frac{1}{c}\kappa \lambda_{(m)}&=\frac{(4\tau_m+n(n-4))\tau_m}{2\tau_m+(n-1)(n-4)}\frac{h'(R)}{h^3(R)}.
\end{align*}
For $n=3$ and $m\geq 2$, we have
\begin{align*}
\xi_{(m)}&\geq \frac{1}{c}\kappa \lambda_{(m)}=\frac{(4\tau_m-13)\tau_m}{2\tau_m-6}\frac{h'(R)}{h^3(R)}.
\end{align*}

So we finish the proof of the inequality parts of Theorem~\ref{thm2} for these two cases.

Finally, let $n\geq 4$ and $m=1$, and assume that the equality in \eqref{bound4} holds. So all the inequalities in this subsection become equalities. Then it is straightforward to check that $h'(r)\equiv 1$, and $\psi(r)=r(r^2-3R^2)$ up to a constant multiple, which is exactly the radial part of a first nontrivial eigenfunction for the Euclidean ball $B_R$ (see \cite[Theorem~1.5]{XW18}). So the warped product manifold is isometric to the Euclidean ball with radius $R$, and the proof is complete.

\subsection{Discussion on $3$-dimensional case}

Let us briefly discuss the higher dimensional case $n\geq 3$ and $m=1$ by use of the approach in Section~\ref{sec4.2}. In this case, letting $\tilde{b}(s)=\tilde{\psi}(r(s))$ and $b(s)=\psi(r(s))$, we need to solve
\begin{align*}
&\tilde{b}''(s)-(n-1)h^{2(n-2)}(r(s))\tilde{b}(s)=0,\\
&b''(s)-(n-1)h^{2(n-2)}(r(s))b(s)=h^{2(n-1)}(r(s))\tilde{b}(s).
\end{align*}
Assume that $\psi_1(s)>0$ and $\psi_2(s)>0$ are two fundamental solutions of the first equation which satisfy
\begin{align*}
&\psi_1(-\infty)=0,\quad \psi_1'(s)>0, \\
&\psi_2(-\infty)=+\infty,\quad  \psi_2'(s)<0.
\end{align*}
Moreover, assume that $\psi_1$ and $\psi_2$ have good rate of decay or growth at infinity such that the intermediate argument as in the case $n=2$ still works. Then we can check that finally the problem reduces to verifying
\begin{align*}
&G(s):=\frac{(\psi_1'(s))^2h^{4-n}(r(s))}{h'(r(s))}-(n+2)\int_{-\infty}^{s} (\psi_1(t))^2h^{2(n-1)}(r(t))dt\geq 0
\end{align*}
for $s=s_0$.

Note that $G(-\infty)\geq 0$. So we may need to prove $G'(s)\geq 0$ for $s\leq s_0$. Now we have
\begin{align*}
G'(s)&=\frac{2\psi_1'(s)\psi_1''(s)h^{4-n}(r(s))}{h'(r(s))}+(4-n)\frac{(\psi_1'(s))^2h^{4-n}(r(s))h'(r(s))}{h'(r(s))}\\
&-\frac{(\psi_1'(s))^2h^{5-n}(r(s))h''(r(s))}{(h'(r(s)))^2}-(n+2)(\psi_1(s))^2h^{2(n-1)}(r(s))\\
&\geq 2(n-1)\psi_1'\psi_1 h^n+(4-n)(\psi_1')^2h^{4-n}-(n+2)(\psi_1)^2h^{2(n-1)}.
\end{align*}

Of course when $n=2$, we get $\psi_1(s)=e^s$ and so $G'(s)\geq 0$. However, for $n\geq 3$, it seems hard to show that $G'(s)\geq 0$, which indicates that the case $n=2$ is quite special.

\subsection{The case $Ric_g\leq 0$}\label{sec4.5}
 One natural question is to consider Theorem~\ref{thm2} for the case $Ric_g\leq 0$. Assume $Ric_g\leq 0$. By checking the proof of Proposition~\ref{prop3}, we may still assume that all the eigenfunctions are of the form $\psi(r)\omega(p)$ with separate variables, and we see that the energies $\beta_k$ are still discrete and correspond to some eigenvalues. However, we are unable to prove that $\beta_k$ is monotone in $m(k)$. In other words, if $k_1$ and $k_2$ are such that $m(k_1)>m(k_2)$, we do not know whether $\beta_{k_1}>\beta_{k_2}$ or not. This means that we can no longer determine the order of these eigenvalues $\beta_k$.

Regardless of this point, if we still denote by $\xi_{(m)}$ the eigenvalue corresponding to an eigenfunction $\psi(r)\omega(p)$ with the spherical harmonic $\omega$ being of degree $m$, then we can prove an optimal upper bound
\begin{equation*}
\xi_{(m)}\leq 2m^2(m+1)\frac{h'(R)}{h^3(R)},
\end{equation*}
for $n=2$ and $m\geq 1$. This can be checked by using $h''(r)\geq 0$ and $h'(r)\geq 1$ in Section~\ref{sec4.2} instead of $h''(r)\leq 0$ and $0<h'(r)\leq 1$.

\section{Proof of Theorem~\ref{thm3}}\label{sec5}

The fourth-order Steklov eigenvalue problem we are concerned with in this section is:
\begin{equation}\label{S5.1}
\begin{cases}
\Delta^2 \varphi=0,\quad &\text{in }M,\\
\varphi=0,\quad \Delta \varphi=\eta  \dfrac{\pt \varphi}{\pt \nu},\quad & \text{on }\pt M.
\end{cases}
\end{equation}
Its $k$th eigenvalue $\eta_k$ has the variational characterization:
\begin{equation}\label{V5.2}
\eta_k=\inf_{\substack{ \varphi\in H^2(M),\: \varphi|_{\pt M}=0,\:\pt_\nu \varphi|_{\pt M}\neq 0\\  \int_{\pt M}\varphi \varphi_i da_g=0,\: i=0,\cdots, k-1}}\frac{\int_M(\Delta \varphi)^2dv_g}{\int_{\pt M} (\pt_\nu \varphi)^2da_g},
\end{equation}
where $\varphi_i$ is the $i$th eigenfunction.

\subsection{The characterization of the Steklov eigenfunctions}
Similar to Proposition~\ref{prop1}, we may characterize the eigenfunctions of the problem~\eqref{S5.1} as follows.
\begin{prop}\label{prop4}
For the warped product manifold $M$ in Theorem~\ref{thm3}, the first eigenfunction $\varphi_0(r,p)$ is given by $\psi_0(r)$ up to a constant multiple, where $\psi_0$ is a nontrivial solution of the ODE
\begin{equation}\label{eq5.3}
\begin{cases}
\dfrac{1}{h^{n-1}}\dfrac{d}{dr}(h^{n-1}\dfrac{d}{dr}\psi)=\tilde{\psi},&\\
\dfrac{1}{h^{n-1}}\dfrac{d}{dr}(h^{n-1}\dfrac{d}{dr}\tilde{\psi})=0,&\\
\psi'(0)=0,\quad \psi(R)=0,\quad \tilde{\psi}'(0)=0.&
\end{cases}
\end{equation}
Any higher-order eigenfunction $\varphi$ of the problem \eqref{S5.1} can be written as $\varphi(r,p)=\psi(r)\omega(p)$, where $\omega$ is a spherical harmonic on $\SS^{n-1}$ of some degree $m\geq 1$, i.e.,
\begin{equation*}
-\Delta_{\SS^{n-1}}\omega=\tau_m\omega \text{ on }\SS^{n-1},\quad \tau_m=m(n-2+m),
\end{equation*}
and $\psi$ is a nontrivial solution of the ODE
\begin{equation}\label{eq5.4}
\begin{cases}
\dfrac{1}{h^{n-1}}\dfrac{d}{dr}(h^{n-1}\dfrac{d}{dr}\psi)-\dfrac{\tau_{m}\psi}{h^2}=\tilde{\psi},&\\
\dfrac{1}{h^{n-1}}\dfrac{d}{dr}(h^{n-1}\dfrac{d}{dr}\tilde{\psi})-\dfrac{\tau_m\tilde\psi}{h^2}=0,&\\
\psi(0)=0,\quad  \psi(R)=0,\quad \tilde{\psi}(0)=0.&
\end{cases}
\end{equation}
For any nontrivial solution $\psi$ to the ODE \eqref{eq5.3} or \eqref{eq5.4}, the $m$th eigenvalue $\eta_{(m)}$ ($m\geq 0$) without counting multiplicity is given by $\eta_{(m)}=\tilde{\psi}(R)/\psi'(R)$.
\end{prop}
\begin{proof}
Take $\{\omega_k\}$, $k=0,1,\cdots$, to be a complete orthonormal basis of $L^2(\SS^{n-1})$ as in Proposition~\ref{prop1}. Let $\psi_0$ be the solution of \eqref{eq5.3} with additional requirement $\psi'(R)=1$. For $k\geq 1$, let $\psi_k$ and $\tilde{\psi}_k$ solve \eqref{eq5.4} with additional condition $\psi'(R)=1$ and with $m=m(k)$. Define $\varphi_k(r,p)=\psi_k(r)\omega_k(p)$, $k\in \N$. We claim that
\begin{equation*}
{span} \{\varphi_k,k\in \N\}=\{\varphi\in C^\infty(B_R):\Delta^2\varphi=0 \text{ in } B_R,\:  \varphi=0 \text{ on } \pt B_R\}.
\end{equation*}

To prove the claim, first we can check directly that $\Delta^2 \varphi_k=0$ in $B_R$ and $ \varphi_k=0$ on $\pt B_R$.

Now take any smooth $\bar{\varphi}$ with $\Delta^2\bar{\varphi}=0$ in $B_R$ and $ \bar{\varphi}=0$ on $\pt B_R$. Since the basis $\{\omega_k\}$, $k=0,1,\dots$, for $L^2(\SS^{n-1})$ is complete, we can first decompose
\begin{equation*}
\pt_\nu\bar\varphi|_{\pt B_R}=\sum_{k=0}^\infty c_k \omega_k,
\end{equation*}
for a sequence of constants $c_k$. Then the function $f:=\bar{\varphi}-\sum_{k=0}^\infty c_k \varphi_k$ satisfies
\begin{equation*}
\Delta^2 f=0, \text{ in }B_R,\quad f=0 \text{ and } \pt_\nu f=0, \text{ on }\pt B_R,
\end{equation*}
which leads to $f=0$, or $\bar{\varphi}=\sum_{k=0}^\infty c_k \varphi_k$. So we have proved the claim. In particular, any eigenfunction $\varphi$ can be written as
\begin{equation*}
\varphi=\sum_{k=0}^\infty c_k \varphi_k.
\end{equation*}
It follows that $\varphi_k$, $k\geq 0$ are all the eigenfunctions with corresponding eigenvalues given by $\tilde{\psi}_k(R)/\psi_k'(R)$. Note that if $k_1$ and $k_2$ are such that $m(k_1)=m(k_2)$, then $\varphi_{k_1}$ and $\varphi_{k_2}$ correspond to the same eigenvalue, since $\psi_{k_1}=\psi_{k_2}$. Therefore, $\varphi_0$ corresponds to an eigenvalue of multiplicity one; $\varphi_k$, $1\leq k\leq n$, correspond to an eigenvalue of multiplicity $n$; etc. Since the first eigenvalue is simple (see \cite[Theorem~1]{BGM06} or \cite{RS15}), we know $\eta_{(0)}=\eta_0=\tilde{\psi}_0(R)/\psi_0'(R)$. For the higher-order eigenvalues $\eta_{(m)}$, $m\geq 1$, we can determine them just as in Proposition~\ref{prop3}. Alternatively, we may use the Reilly's formula to prove it. Let $\varphi(r,p)=\psi(r)\omega(p)$ be a smooth function on $B_R$ with $ \varphi=0$ along $\pt B_R$, where $\omega$ is a spherical harmonic of degree $m$. Applying Proposition~\ref{prop-A} in the Appendix to the Reilly's formula~\eqref{Rei} for $\varphi$, we get
\begin{align*}
\int_{B_R}&(\Delta \varphi)^2=\int_{B_R} |\nabla^2\varphi|^2+\int_{B_R} Ric(\nabla\varphi,\nabla\varphi)+\int_{\pt B_R}H(\pt_\nu \varphi)^2\\
&=\int_0^R \left((\psi'')^2+(n-1)\frac{(\psi')^2(h')^2}{h^2}\right.\\
&\left.+\tau_m\frac{2}{h^2}\left((\psi')^2+\psi^2\frac{(h')^2}{h^2}-3\frac{\psi\psi'h'}{h}\right)+\frac{\psi^2}{h^4}\tau_m(\tau_m-n+2)\right)h^{n-1}\\
&-\int_0^R  \left((n-1)\frac{h''}{h}(\psi')^2+\tau_m\left(\frac{h''}{h}-(n-2)\frac{1-(h')^2}{h^2}\right)\frac{\psi^2}{h^2}\right)h^{n-1}\\
&+(n-1)\frac{h'(R)}{h(R)}(\psi'(R))^2h^{n-1}(R).
\end{align*}
Using $(h')^2\leq 1$, we can check directly that the above formula is strictly increasing in $m\geq 1$. In view of this fact, after further necessary arguments we can determine the order of all the energies
\begin{equation*}
\beta_{k}:=\frac{\tilde{\psi_k}(R)}{\psi_k'(R)}=\dfrac{\int_{B_R}(\Delta (\psi_k \omega_k))^2 dv_g}{\int_{\pt B_R} (\pt_\nu(\psi_k \omega_k))^2 da_g},\quad k\geq 1.
\end{equation*}
The details in this approach are left to interested readers. In either way we can finish the proof of Proposition~\ref{prop4}.

\end{proof}

By Proposition~\ref{prop4}, it suffices to prove that
\begin{equation*}
\eta_{(m)}=\frac{\tilde{\psi}(R)}{\psi'(R)}=\dfrac{\int_{B_R}(\Delta (\psi \omega))^2 dv_g}{\int_{\pt B_R} (\pt_\nu(\psi \omega))^2 da_g}
\end{equation*}
has the lower bound in each case of Theorem~\ref{thm3}, where $\psi$ and $\tilde{\psi}$ solve \eqref{eq5.3} or \eqref{eq5.4} and $\omega$ is a spherical harmonic of degree $m$.

\subsection{Proof of Theorem~\ref{thm3} for $n=2$}\label{sec5.2}

The proof proceeds as in Section~\ref{sec4.2}. Here we only sketch it.

Let $n=2$ and $m\geq 1$. Using the change of variable
\begin{equation*}
s(r):=\int_{R/2}^r\frac{dt}{h(t)},
\end{equation*}
and setting $\tilde{b}(s):=\tilde{\psi}(r(s))$ and $b(s):=\psi(r(s))$, we need to solve
\begin{align*}
&\tilde{b}''(s)-m^2\tilde{b}(s)=0,\\
&b''(s)-m^2b(s)=\tilde{b}(s)h^2(r(s)),
\end{align*}
with boundary conditions $\tilde{b}(-\infty)=0$, $b(-\infty)=0$ and $b(s_0)=0$.

The solution of the above ODE up to a constant multiple is given by
\begin{align*}
\tilde b(s)&=e^{ms},\\
b(s)&=\frac{1}{2m}e^{ms}\int_{s_0}^s h^2(r(t))dt+\frac{1}{2m}e^{m(s-2s_0)}\int_{-\infty}^{s_0}e^{2mt}h^2(r(t))dt\\
&-\frac{1}{2m}e^{-ms}\int_{-\infty}^se^{2mt}h^2(r(t))dt.
\end{align*}
In particular, we obtain
\begin{equation*}
b'(s_0)=e^{-ms_0}\int_{-\infty}^{s_0}e^{2mt}h^2(r(t))dt.
\end{equation*}
Since $\psi'(r)=b'(s)s'(r)=b'(s)/h(r)$, to prove Theorem~\ref{thm3} for $n=2$ it suffices to prove
\begin{equation*}
\frac{e^{ms(R)}h(R)}{b'(s(R))}\geq 2(m+1)\frac{h'(R)}{h(R)},\text{ or } e^{ms_0}\geq 2(m+1)b'(s_0)\frac{h'(r(s_0))}{h^2(r(s_0))}.
\end{equation*}
Then we need to check that the function
\begin{equation*}
G(s):=\frac{e^{2ms}h^2(r(s))}{h'(r(s))}-2(m+1)\int_{-\infty}^{s}e^{2mt}h^2(r(t))dt
\end{equation*}
satisfies $G(s_0)\geq 0$, which follows from Lemma~\ref{lem-G}.

So we have the lower bound \eqref{bound5} for $n=2$ and $m\geq 1$ in Theorem~\ref{thm3}. Moreover, when the equality holds, we must have $h''(r)\equiv 0$ and $h'(r)\equiv 1$, which means $h(r)=r$. So we complete the proof.



\subsection{Proof of Theorem~\ref{thm3} for $n\geq 3$}\label{sec5.3}

Let $\varphi(r,p)=\psi(r)\omega(p)$ be an eigenfunction corresponding to $\eta_{(m)}$ with $m\geq 1$. So we have
\begin{align*}
&\int_{B_R}(\Delta \varphi)^2=\eta_{(m)}\int_{\pt B_R} (\pt_\nu\varphi)^2,\\
&\int_{\pt B_R}H(\pt_\nu \varphi)^2=(n-1)\kappa \int_{\pt B_R} (\pt_\nu\varphi)^2,
\end{align*}
where $\kappa=h'(R)/h(R)$ denotes the principal curvature of the boundary $\pt B_R$.

By the Reilly's formula~\eqref{Rei}, for fixed $c\in (0,1)$, we obtain
\begin{align*}
&(c\eta_{(m)}-(n-1)\kappa)\int_{\pt B_R} (\pt_\nu\varphi)^2\\
&=\int_{B_R} |\nabla^2\varphi|^2-(1-c)\int_{B_R}(\Delta \varphi)^2+\int_{B_R} Ric(\nabla\varphi,\nabla\varphi).
\end{align*}
Again our goal is to find $c$ as small as possible such that the right-hand side of the above formula is nonnegative, which can be achieved as in Section~\ref{sec4.3}.

So according to the argument in Section~\ref{sec4.3}, for $n\geq 4$ and $m\geq 1$, we have
\begin{align*}
\eta_{(m)}\geq \frac{1}{c}(n-1)\kappa &=\frac{(4\tau_m+n(n-4))(n-1)}{2\tau_m+(n-1)(n-4)}\frac{h'(R)}{h(R)}.
\end{align*}
For $n=3$ and $m\geq 2$, we have
\begin{align*}
\eta_{(m)}&\geq \frac{1}{c}2\kappa =\frac{4\tau_m-13}{\tau_m-3}\frac{h'(R)}{h(R)}.
\end{align*}

Hence we finish the proof of the inequality parts of Theorem~\ref{thm3} for these two cases.

Finally, let $n\geq 4$ and $m=1$, and assume that the equality in \eqref{bound7} holds. So all the inequalities along the corresponding argument in Section~\ref{sec4.3} become equalities. Then it is straightforward to check that $h'(r)\equiv 1$, and $\psi(r)=r(r^2-R^2)$ up to a constant multiple, which is exactly the radial part of the corresponding eigenfunction for the Euclidean ball $B_R$ (see \cite[Theorem~1.3]{FGW05}). So the warped product manifold is isometric to the Euclidean ball with radius $R$, and the proof is complete.

\subsection{The case $Ric_g\leq 0$}\label{5.4}

Assume $Ric_g\leq 0$. As explained in Section~\ref{sec4.5}, in this case we can no longer determine the order of the eigenvalues $\beta_k$ (the energies), except the first eigenvalue which is a simple one.

If we still denote by $\eta_{(m)}$ the eigenvalue corresponding to an eigenfunction $\psi(r)\omega(p)$ with the spherical harmonic $\omega$ being of degree $m$, then we can prove an optimal upper bound
\begin{equation*}
\eta_{(m)}\leq 2(m+1)\frac{h'(R)}{h(R)},
\end{equation*}
for $n=2$ and $m\geq 0$. This can be checked by using $h''(r)\geq 0$ and $h'(r)\geq 1$ in Section~\ref{sec5.2} instead of $h''(r)\leq 0$ and $0<h'(r)\leq 1$.

\section*{Appendix}

\setcounter{equation}{0}
\renewcommand{\theequation}{A.\arabic{equation}}

Here we present some computation results which are needed in the main part of this paper.
\begin{prop}\label{prop-A}
Let $M^n=[0,R)\times \SS^{n-1}$ be an $n$-dimensional ($n\geq 2$) smooth Riemannian manifold equipped with the warped product metric
\begin{equation*}
g=dr^2+h^2(r)g_{\SS^{n-1}},
\end{equation*}
where the warping function $h$ satisfies Assumption (A). Assume that $\varphi(r,p)=\psi(r)\omega(p)$, $r\in [0,R)$, $p\in \SS^{n-1}$, is a smooth function on $M$, where $\omega$ is a spherical harmonic of degree $m$, i.e., $-\Delta_{\SS^{n-1}}\omega=\tau_m\omega$, $\tau_m=m(n-2+m)$. Then we have
\begin{align*}
\int_{B_R}&|\nabla^2\varphi|^2=\int_0^R \left((\psi'')^2+(n-1)\frac{(\psi')^2(h')^2}{h^2}\right.\\
&\left.+\tau_m\frac{2}{h^2}\left((\psi')^2+\psi^2\frac{(h')^2}{h^2}-3\frac{\psi\psi'h'}{h}\right)+\frac{\psi^2}{h^4}\tau_m(\tau_m-n+2)\right)h^{n-1},\\
\int_{B_R}&(\Delta \varphi)^2=\int_0^R \left(\psi''+(n-1)\psi'\frac{h'}{h}-\tau_m\frac{\psi}{h^2}\right)^2h^{n-1},
\end{align*}
and
\begin{align*}
&\int_{B_R} Ric(\nabla\varphi,\nabla\varphi)=\\
&-\int_0^R  \left((n-1)\frac{h''}{h}(\psi')^2+\tau_m\left(\frac{h''}{h}-(n-2)\frac{1-(h')^2}{h^2}\right)\frac{\psi^2}{h^2}\right)h^{n-1}.
\end{align*}
\end{prop}
\begin{proof}

In the proof we write $\tau$ for $\tau_m$ for the sake of simplicity. Take a local coordinate system $\{\theta^1,\dots,\theta^{n-1}\}$ for $\SS^{n-1}$. Let ${e_{ij}}$ be the components of the metric $g_{\SS^{n-1}}$ with respect to this coordinate, i.e.,
\begin{equation*}
g_{\SS^{n-1}}=e_{ij}d\theta^id\theta^j,
\end{equation*}
and so
\begin{equation*}
g=dr^2+h^2e_{ij}d\theta^i d\theta^j.
\end{equation*}

First we have
\begin{equation*}
\nabla\varphi=\psi' \omega \pt_r+\frac{\psi}{h^2}e^{ij}\omega_i\pt_j,
\end{equation*}
where $\pt_j$ denotes the vector field $\dfrac{\pt}{\pt \theta^j}$.

So we get
\begin{align*}
\nabla^2\varphi(\pt_r,\pt_r)&=g(\nabla_{\pt_r}(\nabla\varphi),\pt_r)=\psi''\omega.
\end{align*}

Next we deduce
\begin{align*}
\nabla^2\varphi(\pt_r,\pt_i)&=g(\nabla_{\pt_r}(\nabla\varphi),\pt_i)=g(\nabla_{\pt_r}(\psi' \omega \pt_r+\frac{\psi}{h^2}e^{kl}\omega_k\pt_l),\pt_i)\\
&=(\frac{\psi'}{h^2}-2\frac{h'\psi}{h^3})h^2\omega_i+\psi\frac{h'}{h}\omega_i=(\psi'-\psi \frac{h'}{h})\omega_i,
\end{align*}
where we have used
\begin{equation*}
\nabla_{\pt_r}\pt_l=\nabla_{\pt_l}\pt_r=\frac{h'}{h}\pt_l.
\end{equation*}

Finally we obtain
\begin{align*}
\nabla^2&\varphi(\pt_i,\pt_j)=g(\nabla_{\pt_i}(\nabla\varphi),\pt_j)=g(\nabla_{\pt_i}(\psi' \omega \pt_r+\frac{\psi}{h^2}e^{kl}\omega_k\pt_l),\pt_j)\\
&=\psi'\omega hh'e_{ij}+\psi \nabla^S_{ij}\omega,
\end{align*}
where $\nabla^S$ denotes the connection on $\SS^{n-1}$.

As a consequence, we derive
\begin{align*}
|\nabla^2\varphi|^2&=(\varphi_{rr})^2+2g^{ij}\varphi_{ri}\varphi_{rj}+h^{-4}e^{ij}e^{kl}\varphi_{ik}\varphi_{jl}\\
&=(\psi'')^2\omega^2+2h^{-2}(\psi'-\psi \frac{h'}{h})^2|\nabla^S \omega|^2\\
&+h^{-4}\left((\psi'\omega h h')^2(n-1)+2\psi\psi'\omega hh'\Delta^S \omega+\psi^2|(\nabla^S)^2 \omega|^2\right).
\end{align*}
Therefore we have
\begin{align*}
\int_{B_R}&|\nabla^2\varphi|^2=\int_0^R \left((\psi'')^2+(n-1)\frac{(\psi')^2(h')^2}{h^2}\right.\\
&\left.+\tau\frac{2}{h^2}\left((\psi')^2+\psi^2\frac{(h')^2}{h^2}-3\frac{\psi\psi'h'}{h}\right)+\frac{\psi^2}{h^4}\int_{\SS^{n-1}}|(\nabla^S)^2 \omega|^2da\right)h^{n-1}.
\end{align*}
Recall the Bochner's formula \eqref{Boc} on $\SS^{n-1}$,
\begin{equation*}
\frac{1}{2}\Delta^{S}(|\nabla^S \omega|^2)=|(\nabla^S)^2 \omega|^2+g_{\SS^{n-1}}(\nabla^S \omega,\nabla^S(\Delta^S \omega))+\mathrm{Ric}_{\SS^{n-1}}(\nabla^S \omega,\nabla^S \omega).
\end{equation*}
Note that $\Delta^S \omega=-\tau\omega$ and $\mathrm{Ric}_{\SS^{n-1}}=(n-2)g_{\SS^{n-1}}$. So after integration we have
\begin{equation*}
\int_{\SS^{n-1}}|(\nabla^S)^2 \omega|^2da=\tau(\tau-n+2).
\end{equation*}
Thus we obtain
\begin{align*}
\int_{B_R}&|\nabla^2\varphi|^2=\int_0^R \left((\psi'')^2+(n-1)\frac{(\psi')^2(h')^2}{h^2}\right.\\
&\left.+\tau\frac{2}{h^2}\left((\psi')^2+\psi^2\frac{(h')^2}{h^2}-3\frac{\psi\psi'h'}{h}\right)+\frac{\psi^2}{h^4}\tau(\tau-n+2)\right)h^{n-1}.
\end{align*}

Next we have
\begin{align*}
(\Delta \varphi)^2&=\left(\psi''+(n-1)\psi'\frac{h'}{h}-\tau\frac{\psi}{h^2}\right)^2\omega^2,
\end{align*}
which implies
\begin{align*}
&\int_{B_R}(\Delta \varphi)^2=\int_0^R \left(\psi''+(n-1)\psi'\frac{h'}{h}-\tau\frac{\psi}{h^2}\right)^2h^{n-1}.
\end{align*}

Lastly recall
\begin{align*}
Ric_g&=-\left(\frac{h''}{h}-(n-2)\frac{1-(h')^2}{h^2}\right)g-(n-2)\left(\frac{h''}{h}+\frac{1-(h')^2}{h^2}\right)dr^2\\
&=-(n-1)\frac{h''}{h}dr^2-\left(\frac{h''}{h}-(n-2)\frac{1-(h')^2}{h^2}\right)h^2g_{\SS^{n-1}}.
\end{align*}
So
\begin{align*}
&Ric(\nabla\varphi,\nabla\varphi)=-(n-1)\frac{h''}{h}(\psi')^2\omega^2-\left(\frac{h''}{h}-(n-2)\frac{1-(h')^2}{h^2}\right)\frac{\psi^2}{h^2}|\nabla^S \omega|^2,
\end{align*}
which yields
\begin{align*}
&\int_{B_R} Ric(\nabla\varphi,\nabla\varphi)=\\
&-\int_0^R  \left((n-1)\frac{h''}{h}(\psi')^2+\tau\left(\frac{h''}{h}-(n-2)\frac{1-(h')^2}{h^2}\right)\frac{\psi^2}{h^2}\right)h^{n-1}.
\end{align*}

So the proof is complete.

\end{proof}


\bibliographystyle{Plain}

\end{document}